\documentclass[11pt]{article}
\usepackage[T1]{fontenc}
\usepackage{lmodern,amsmath,amsthm,amsfonts,amssymb,graphicx,float,wrapfig,calc,microtype,thmtools,underscore,mathtools,stmaryrd,mathabx}
\usepackage[usenames,dvipsnames,svgnames,table]{xcolor}
\usepackage[unicode=true]{hyperref}
\usepackage{tikz}
\hypersetup{ 
colorlinks,
linkcolor={black},
citecolor={black},
urlcolor={blue!60!black},
pdftitle={Reduced bandwidth: a qualitative strengthening of twin-width in minor-closed classes (and beyond)},
pdfauthor={\'Edouard Bonnet, O-Joung Kwon, David~R.~Wood}} 
\usepackage[noabbrev,capitalise]{cleveref}
\crefname{lem}{Lemma}{Lemmas}
\crefname{thm}{Theorem}{Theorems}
\crefname{cor}{Corollary}{Corollaries}
\crefname{prop}{Proposition}{Propositions}
\crefname{conj}{Conjecture}{Conjectures}
\crefname{openproblem}{Open Problem}{Open Problems}
\crefformat{equation}{(#2#1#3)}
\Crefformat{equation}{Equation #2(#1)#3}

\usepackage[shortlabels]{enumitem}
\setlist[itemize]{topsep=0ex,itemsep=0ex,parsep=0ex}
\setlist[enumerate]{topsep=0ex,itemsep=0ex,parsep=0ex}
\setlist[description]{topsep=0ex,itemsep=0ex,parsep=0ex}
\newcommand{\defn}[1]{\textcolor{Maroon}{\emph{#1}}}

\newcommand{\GG}{\mathcal{G}}

\newcommand{\reduced}[1]{{#1}^{\downarrow}}
\newcommand\abs[1]{\lvert #1\rvert}

\newcommand{\cT}{\mathcal{T}}
\newcommand{\cB}{\mathcal{B}}
\newcommand{\cP}{\mathcal{P}}

\newcommand{\red}{\operatorname{red}}
\newcommand{\blowup}{\,\rotatebox{90}{$\bowtie$}}
\usepackage[longnamesfirst,numbers,sort&compress]{natbib}
\makeatletter
\def\NAT@spacechar{~}
\makeatother
\usepackage[margin=29mm]{geometry}
\allowdisplaybreaks

\DeclarePairedDelimiter{\floor}{\lfloor}{\rfloor}

\renewcommand{\ge}{\geqslant}
\renewcommand{\le}{\leqslant}
\renewcommand{\geq}{\geqslant}
\renewcommand{\leq}{\leqslant}
\DeclareMathOperator{\dist}{dist}
\DeclareMathOperator{\diam}{diam}

\DeclareMathOperator{\col}{col}
\DeclareMathOperator{\reach}{reach}
\DeclareMathOperator{\tw}{tw}

\DeclareMathOperator{\tww}{tww}

\DeclareMathOperator{\cw}{cw}

\DeclareMathOperator{\pw}{pw}
\DeclareMathOperator{\bw}{bw}
\DeclareMathOperator{\boolw}{blw}

\renewcommand{\thefootnote}{\fnsymbol{footnote}}
\allowdisplaybreaks
\theoremstyle{plain}
\newtheorem{thm}{Theorem}
\newtheorem{lem}[thm]{Lemma}
\newtheorem{cor}[thm]{Corollary}

\theoremstyle{definition}

\newcommand{\NN}{\mathbb{N}}

\begin{document}

\author{\'Edouard Bonnet\,\footnotemark[3] \qquad 
O-joung Kwon\,\footnotemark[4] \qquad David~R.~Wood\,\footnotemark[5]}

\footnotetext[3]{CNRS, Lyon, France (\texttt{edouard.bonnet@ens-lyon.fr}). The first author was supported by the ANR projects TWIN-WIDTH (ANR-21-CE48-0014) and Digraphs (ANR-19-CE48-0013).}

\footnotetext[4]{Department of Mathematics, Hanyang University, and Discrete Mathematics Group, Institute for Basic Science (IBS), South Korea. The second author was supported by the National Research Foundation of Korea (NRF) grant funded by the Ministry of Education (No. NRF-2021K2A9A2A11101617 and No. RS-2023-00211670), and by the Institute for Basic Science (IBS-R029-C1). (\texttt{ojoungkwon@hanyang.ac.kr}).}

\footnotetext[5]{School of Mathematics, Monash   University, Melbourne, Australia  (\texttt{david.wood@monash.edu}). Research supported by the Australian Research Council and NSERC.}

\sloppy

\title{\textbf{Reduced bandwidth: a qualitative\\ strengthening of twin-width in\\ minor-closed classes (and beyond)}}

\maketitle

\begin{abstract}
In a \defn{reduction sequence} of a graph, vertices are successively identified until the graph has one vertex. At each step, when identifying $u$ and $v$, each edge incident to exactly one of $u$ and $v$ is coloured red. Bonnet, Kim, Thomass\'e and Watrigant [\emph{J. ACM} 2022] defined the \defn{twin-width} of a graph $G$ to be the minimum integer $k$ such that there is a reduction sequence of $G$ in which every red graph has maximum degree at most $k$. For any graph parameter $f$, we define the \defn{reduced~$f$} of a graph~$G$ to be the minimum integer~$k$ such that there is a reduction sequence of $G$ in which every red graph has $f$ at most~$k$. Our focus is on graph classes with bounded reduced bandwidth, which implies and is stronger than bounded twin-width (reduced maximum degree). We show that every proper minor-closed class has bounded reduced bandwidth, which is qualitatively stronger than an analogous result of Bonnet et al.\ for bounded twin-width. In many instances, we also make quantitative improvements. For example, all previous upper bounds on the twin-width of planar graphs were at least $2^{1000}$. We show that planar graphs have reduced bandwidth at most $466$ and twin-width at most $583$. Our bounds for graphs of Euler genus $\gamma$ are $O(\gamma)$. Lastly, we show that fixed powers of graphs in a proper minor-closed class have bounded reduced bandwidth (irrespective of the degree of the vertices). In particular, we show that map graphs of Euler genus $\gamma$ have reduced bandwidth $O(\gamma^4)$. Lastly, we separate twin-width and reduced bandwidth by showing that any infinite class of expanders excluding a fixed complete bipartite subgraph has unbounded reduced bandwidth, while there are bounded-degree expanders with twin-width at most 6.
\end{abstract}

\renewcommand{\thefootnote}{\arabic{footnote}}

\newpage

\section{Introduction}
\label{Introduction}

Twin-width is a measure of graph\footnote{We consider simple, finite, undirected graphs $G$ with vertex-set $V(G)$ and edge-set $E(G)$. A \defn{graph class} is a set of graphs closed under isomorphism. A graph class is \defn{hereditary} if it is closed under taking subgraphs. A graph class is \defn{monotone} if it is closed under taking induced subgraphs. A graph $H$ is a \defn{minor} of a graph $G$ if $H$ is isomorphic to a graph obtained from a subgraph of $G$ by contracting edges. A graph class $\GG$ is \defn{proper minor-closed} if $\GG$ is closed under taking minors, and some graph is not in $\GG$.} complexity introduced by \citet{TW-I} (inspired by the work of \citet{MT04} and \citet{GM14}). The topic has attracted widespread interest~\citep{BDHK24,BHJ24,PS23,Przyb23,PSZ22,TW-I,TW-II,TW-III,TW-IV,TW-VI,BNOST24,BH21,BKRTW21,GPT22,AHKO22,SS22,BBD21,BBD22,DGJOR22,GT26,HIRV26,HJ25b,HNST25,ACHO25,BT25,HR25,ACHKO24,BCKKLT22,BH25,BD23,BGOT23,BGTT22}, often motivated by connections to model theory, logic, graph sparsity, fixed parameter tractability, enumerative combinatorics, and permutations.

We start with an informal description of twin-width. Given a graph $G$, choose two vertices $u$ and $v$ in $G$, identify $u$ and $v$ into a single new vertex, insert an edge between this new vertex and a neighbour of $u$ or $v$, and colour the inserted edge red if it is incident to exactly one of $u$ and $v$ in the original graph. Repeat this step until the graph has only one vertex. At each stage, the introduced red edges indicate an `error' in the reduction sequence. The goal is to find a sequence of identifications with small error. Twin-width measures the error by the maximum degree of the red graph (minimised over all reduction sequences). 

To formalise this idea we need the following definitions. A \defn{trigraph} is a triple $G = ( V, E, R )$ where $V$ is a finite set, and $E$ and $R$ are disjoint subsets of $\binom{V}{2}$. Elements of $V$ are \defn{vertices}. Elements of $E\cup R$ are \defn{edges}, edges in $E$ are \defn{black}, and edges in $R$ are \defn{red}. Let $V(G):=V$ and $E(G):=E$ and $R(G):=R$. Let $\widetilde{G}$ be the spanning subgraph of $G$ consisting of the red edges. For distinct vertices  $u,v\in V(G)$, let \defn{$G/u,v$} be the trigraph $(V',E',R')$ with:
\begin{itemize}
\item $V' = ( V \setminus \{u, v\} ) \cup \{w\}$ where $w\not\in V$, 
\item $G-\{u,v\} = ( G/u,v ) - w$, and
\item for all $x\in V\setminus\{u,v\}$:
\begin{itemize}
	\item $wx \in E'$ if and only if $ux \in E$ and $vx \in E$,
	\item $wx \not\in E' \cup R'$ if and only if $ux \not\in E \cup R$ and $vx \not\in E \cup R$, and
	\item $wx \in R'$ otherwise.
\end{itemize}
\end{itemize}
The \defn{underlying graph} (or \defn{total graph}) of a trigraph $G$ is the graph $H$ with $V(H)=V(G)$ and $E(H)=E(G)\cup R(G)$.

A sequence of trigraphs $G_n,G_{n-1},\dots,G_1$ is a \defn{reduction sequence} of $G_n$ (also called \defn{contraction sequence}) if for each $i\in\{2,\dots,n\}$ we have $G_{i-1}=G_i/u,v$ for some $u,v\in V(G_i)$,  and $G_1$ is a trigraph with one vertex. In this case, each `prefix' $G_n, G_{n-1}, \ldots, G_i$ is called a \defn{partial reduction sequence to} $G_i$. A (\defn{partial}) \defn{reduction sequence} of a graph $G$ is a (partial) reduction sequence of the trigraph $(V(G),E(G),\emptyset)$ (with no red edges).

Given a graph $G$, it is natural to ask for a reduction sequence $G_n,\dots,G_1$ such that the red graphs $\widetilde{G}_n,\dots,\widetilde{G}_1$ have desirable properties. For example, a graph has a reduction sequence with no red edges if and only if it is a cograph \citep{TW-I} (and cographs are considered to be particularly well-behaved). \citet{TW-I} cared about the maximum degree of the red graphs. They defined the \defn{twin-width} of a graph $G$, denoted by \defn{$\tww(G)$}, to be the minimum $k\in\NN_0$ such that there is a reduction sequence $G_n,G_{n-1},\dots,G_1$ of $G$ where $\widetilde{G}_i$ has maximum degree at most $k$ for each $i\in\{1,\dots,n\}$. 

This paper studies reduction sequences where the red graph has other properties in addition to bounded maximum degree. For any graph parameter\footnote{A \defn{graph parameter} is a function $f$ such that $f(G)\in\mathbb{N}_0$ for every graph $G$, and $f(G_1)=f(G_2)$ for all isomorphic graphs $G_1$ and $G_2$. Examples of relevance to this paper include 
 maximum degree $\Delta(G)$, 
 bandwidth $\bw(G)$, 
 pathwidth $\pw(G)$, 
 treewidth $\tw(G)$, 
 clique-width $\cw(G)$, 
and boolean-width $\boolw(G)$.
 A graph parameter $f$ is \defn{monotone} if for each $k\in\mathbb{N}_0$ the graph class $\{ G : f(G)\leq k\}$ is monotone. A graph parameter $f$ is \defn{hereditary} if for each $k\in\mathbb{N}_0$ the graph class $\{ G : f(G)\leq k\}$ is hereditary. A graph parameter $f$ is \defn{union-closed} if $f(G\cup H) \leq \max(f(G), f(H))$ for all disjoint graphs $G$ and $H$.} $f$, let \defn{reduced~$f$} be the graph parameter~$\reduced{f}$, where for any graph $G$,  \defn{$\reduced{f}(G)$} is the minimum $k\in\mathbb{N}$ such that there is a reduction sequence $G_n,G_{n-1},\dots,G_1$ of $G$ where $f(\widetilde{G}_i)\leq k$ for each $i\in\{1,\dots,n\}$. So  reduced maximum degree~$\reduced{\Delta}$ is the same as  twin-width.

Every graph has a reduction sequence in which every red graph is a star (just repeatedly identify the centre of the star with any other vertex). So it makes sense to consider graph properties that are unbounded on the class of all stars (such as maximum degree).

This line of research was initiated by \citet{TW-VI}, who considered the following  parameter\footnote{\citet{TW-VI} also considered the total number of edges in the red graph, but with a slightly different notion of reduction sequence in which red loops appear on identified vertices. The resulting parameter is called \defn{total twin-width}.}. Let $\star(G)$ be the maximum number of vertices in a connected component of a graph $G$. Then $\reduced{\star}$ is called the \defn{component-twinwidth} \citep{TW-VI}; here the goal is to find a reduction sequence such that every red graph has small components. \citet{TW-I} proved that every graph $G$ satisfies $\reduced{\star}(G) \leq 2^{\boolw(G)+1}$, which implies $\reduced{\star}(G) \leq 2^{\tw(G)+2}$ and $\reduced{\star}(G) \leq 2^{\cw(G)+1}$, where
$\boolw(G)$, $\tw(G)$, and $\cw(G)$ denote the boolean-width, treewidth and clique-width of $G$, respectively.
Note that $\Delta(G) \leq \star(G)-1$ and thus $\reduced{\Delta}(G) \leq \reduced{\star}(G)-1$.

\citet{TW-I} proved that every proper minor-closed graph class has bounded twin-width (amongst other results). The primary contribution of this paper is a qualitative strengthening of this result, where reduced maximum degree is replaced by reduced bandwidth. The \defn{bandwidth} of a graph $G$, denoted by $\bw(G)$, is the minimum $k\in\mathbb{N}_0$ such that there is an ordering $v_1,\dots,v_n$ of $V(G)$ satisfying $|i-j|\leq k$ for every edge $v_iv_j\in E(G)$. We prove that every proper minor-closed class has bounded reduced bandwidth (\cref{MinorClosedPower}). 

Note that $\Delta(G)\leq2\,\bw(G)$, implying $\reduced{\Delta}(G)\leq 2\,\reduced{\bw}(G)$. Thus, our upper bound on the reduced bandwidth of proper minor-closed classes implies the above-mentioned analogous result for twin-width, and indeed is qualitatively stronger since there are graph classes with bounded maximum degree and unbounded bandwidth. Complete binary trees are a simple example~\citep{CS89}. There are even trees with maximum degree 3, pathwidth 2, and unbounded bandwidth\footnote{\label{QnHere} Let $Q_n$ be the tree consisting of disjoint paths $P_1,\dots,P_n$, each with $n$ vertices, plus an edge joining the first vertex in $P_i$ and the first vertex in $P_{i+1}$ for each $i\in\{1,\dots,n-1\}$. Observe that $Q_n$ has maximum degree 3, pathwidth 2, $n^2$ vertices, diameter less than $3n$, and bandwidth at least $\frac{n}{3}$ (since $|V(G)| \leq \bw(G)\diam(G)+1$ for every graph $G$; see \citep{CS89}).}. Generally speaking, graphs with bounded bandwidth are considered to be particularly well-behaved. Indeed, every graph with bandwidth $k$ is a subgraph of the $k$-th power of a path\footnote{For a graph $G$ and $r\in \mathbb{N}$, the \defn{$r$-th power} of $G$, denoted by \defn{$G^r$}, is the graph with vertex-set $V(G)$ where two vertices $u$ and $v$ are adjacent in $G^r$ if and only if the distance between $u$ and $v$ in $G$ is at most $r$. The 2-nd power of $G$ is called the \defn{square} of $G$.}. 

In many cases, our results are also quantitatively stronger than previous bounds. The improvements for planar graphs are most significant. The previous proofs that planar graphs have bounded twin-width gave no explicit bounds, but it can be seen that all the previous bounds \citep{TW-VI,TW-I} were at least $2^{1000}$. We show that every planar graph has reduced bandwidth at most $466$ and twin-width at most $583$. The proof method generalises for graphs embeddable on any surface\footnote{The \defn{Euler genus} of a surface with $h$ handles and $c$ crosscaps is $2h+c$. The \defn{Euler genus} of a graph $G$ is the minimum Euler genus of a surface in which $G$ embeds without edge crossings. For $\gamma\in\NN_0$, the class of graphs of Euler genus at most $\gamma$ is a proper minor-closed class.}. In particular, we show that every graph with Euler genus $\gamma$ has reduced bandwidth at most $164\gamma+466$ and twin-width at most $205\gamma+583$ (\cref{ReducedBandwidthGenus,TwinwidthGenus}).

A key tool in our proofs are recent product structure theorems, which say that every graph of bounded Euler genus is a subgraph of the strong product of a graph with bounded treewidth and a path; see \cref{ProductStructure} for details. Our results hold for any graph class that has such a product structure, which includes several non-minor-closed graph classes. For example, a graph is \defn{$(\gamma, k)$-planar} if it has a drawing in a surface of Euler genus $\gamma$ such that every edge is involved in at most $k$ crossings (assuming no three edges cross at a single point)~\citep{DEW17}. We prove that every $(\gamma, k)$-planar graph has reduced bandwidth $2^{O(\gamma k)}$ (\cref{gkPlanar}). 

We also strengthen the above-mentioned results by showing that fixed powers of graphs in any proper minor-closed class have bounded reduced bandwidth (\cref{MinorClosedPower}).  Since FO-transductions preserve bounded twin-width~\cite[Section 8]{TW-I}, it was previously known that these graphs have bounded twin-width. Note that powers of sparse graphs can be dense; for instance, $2$-powers of stars are complete graphs. As an example of our results for graphs powers, we consider map graphs, which are well-studied generalisations of graphs embedded in surfaces. We prove that map graphs of Euler genus $\gamma$ have  reduced bandwidth $O(\gamma^4)$. We emphasise there is no dependence on degree, and that these graphs might be dense. 

\cref{Expanders} considers limitations of reduced bandwidth. We show that any infinite class of expander graphs excluding a fixed complete bipartite subgraph has unbounded reduced bandwidth (\cref{ExpanderClass}). This result separates reduced bandwidth from twin-width, since \citet{TW-II} showed there are bounded-degree expanders (thus excluding a fixed complete bipartite subgraph) with twin-width at most 6. The theme of tied and separated parameters is continued in \cref{Tied} where we show that the reduced versions of several natural parameters are separated. We conclude in \cref{OpenProblems} by presenting a number of open problems.

\section{Preliminaries}
\label{Definitions}

Let $\NN:=\{1,2,\dots\}$ and $\NN_0:=\{0,1,\dots\}$. 

For a graph $G$ and $S\subseteq V(G)$, let \defn{$G[S]$} be the \defn{induced} subgraph of $G$ with vertex-set $S$ and edge-set $\{vw\in E(G):v,w\in S\}$. For $F\subseteq E(G)$, let \defn{$G-F$} be the graph obtained from $G$ by removing edges in $F$.
For two graphs $G$ and $H$, let \defn{$G\cup H$} be the graph with vertex-set $V(G)\cup V(H)$ and edge-set $E(G)\cup E(H)$.
A \defn{clique} in $G$ is a (possibly empty) set of pairwise adjacent vertices. For disjoint sets $S, T\subseteq V(G)$, we say that $S$ is \defn{complete} to $T$ in $G$ if every vertex in $S$ is adjacent to every vertex in $T$.

For vertices $v,w\in V(G)$, a \defn{$(v,w)$-path} in $G$ is a path with end-vertices $v$ and $w$. For vertices $v,w\in V(G)$, let \defn{$\dist_G(v,w)$} be the length of a shortest $(v,w)$-path in $G$, and if no such path exists, then we set ${\dist_G(v,w) := \infty}$. For a vertex $v$ in $G$, let \defn{$N_G(v)$} $:=\{w\in V(G):vw\in E(G)\}$ and \defn{$N_G[v]$} $:=\{v\}\cup N_G(v)$. Let \defn{$\deg_G(v)$} $:=|N_G(v)|$, called the \defn{degree} of $v$ in $G$. For each $r\in\NN$,  let \defn{$N^r_G[v]$} $:=\{w\in V(G):\dist_G(v,w)\leq r\}$.
For $S\subseteq V(G)$, let \defn{$N_G(S)$} $:=\bigcup_{v\in S}N_G(v)\setminus S$ and  \defn{$N_G[S]$} $:=\bigcup_{v\in S}N_G[v]$. 

For a vertex $v$ in a trigraph $G$, let \defn{$N_G(v)$} $:=\{w\in V(G):vw\in E(G)\cup R(G)\}$. 

\subsection{Tree-decompositions}

A \defn{tree-decomposition} of a graph $G$ is a pair $(T, \mathcal{B})$ consisting of a tree $T$ and a collection $\mathcal{B}=(B_x\subseteq V(G):x\in V(T))$ of subsets of $V(G)$ (called \defn{bags}) indexed by the nodes of $T$, such that:
\begin{enumerate}[label=(\alph*)]
\item for every edge $uv\in E(G)$, some bag $B_x$ contains both $u$ and $v$, and
\item for every vertex $v\in V(G)$, the set $\{x\in V(T):v\in B_x\}$ induces a non-empty subtree of~$T$.
\end{enumerate}
The \defn{width} of a tree-decomposition is the size of the largest bag minus 1. The \defn{treewidth} \defn{$\tw(G)$} of a graph $G$ is the minimum width of a tree-decomposition of $G$. These definitions are due to \citet{RS-II}. Treewidth is recognised as the most important measure of how similar a given graph is to a tree. Note that a connected graph with at least two vertices has treewidth 1 if and only if it is a tree. 
A \defn{path-decomposition} is a tree-decomposition in which the underlying tree is a path. The \defn{pathwidth} \defn{$\pw(G)$} of a graph $G$ is the minimum width of a path-decomposition of $G$. It is well-known and easily proved that for every graph $G$, 
\[ \tw(G) \leq \pw(G) \leq \bw(G).\]

Let $(T, \mathcal{B}=(B_x:x\in V(T)))$ be a tree-decomposition of a graph $G$. 
The \defn{torso} of a bag $B_u$ is the subgraph obtained from $G[B_u]$ by adding, for each edge $uv\in E(T)$, all edges $xy$ where $x$ and $y$ are distinct vertices in $B_u\cap B_v$. 

A \defn{separation} of a graph $G$ is a pair $(A, B)$ of subsets of $V(G)$ such that $A\cup B=V(G)$ and there is no edge of $G$ between $A\setminus B$ and $B\setminus A$.

We sometimes consider a tree-decomposition $(T, \cB=(B_x:x\in V(T)))$ to be \defn{rooted} at a specific \defn{root bag} $B_r$ for some $r\in V(T)$. In this case, for every node $t\neq r$, let $t'$ be the node of $T$ adjacent to $t$ in $T$ and on the $(r,t)$-path in $T$. We say that $B_{t'}$ is the \defn{parent} of $B_t$ and $B_t$ is a \defn{child} of $B_{t'}$. A bag $B_x$ is a \defn{descendant} of a bag $B_y$ if $y$ lies on the $(r,x)$-path in $T$. 

Let $t$ be a node of $T$, and let $\{t_1, \ldots, t_d\} $ be a set of children of $t$. Let $C$ be the union of $B_t$ and every bag that is a descendant of $B_{t_i}$ for some $i\in \{1, 2, \ldots, d\}$. Let $D:=(V(G)\setminus C)\cup B_t$. As illustrated in \cref{fig:rootedsep}, $(C, D)$ is a separation of $G$ with $C\cap D=B_t$, said to be a \defn{rooted separation at $B_t$} and a \defn{rooted separation from $(T, \cB)$}.

\begin{figure}[!ht]
\centerline{\includegraphics[scale=0.35]{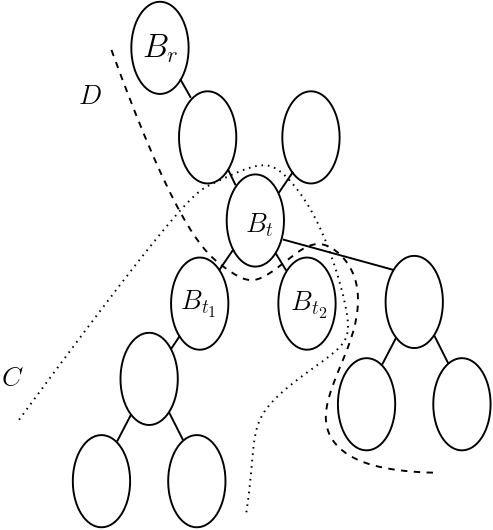} }
\caption{A rooted separation $(C,D)$ in a rooted tree-decomposition with root bag $B_r$.}
\label{fig:rootedsep}
\end{figure}

For a non-root degree-1 node $t$ of $T$, we say that $B_t$ is a \defn{leaf bag};  every other bag is said to be \defn{internal}. Note that a root bag is always internal.

For $k,q\in\NN$ with $q\ge k+1$, a rooted tree-decomposition $(T, \cB)$ is \defn{$(k,q)$-rooted} if:
\begin{itemize}
 \item the root bag is empty,
 \item every internal bag has at most $k+1$ vertices, and
 \item for every leaf bag $B$ with parent $B'$, $\abs{B\setminus B'}\le q$.
\end{itemize}
A rooted tree-decomposition $(T, \cB)$ is \defn{$(k, \infty)$-rooted} if the root bag is empty, and  every internal bag has at most $k+1$ vertices (so  leaf bags can be arbitrarily large). 

\subsection{Sparsity}

For $d\in\NN_0$, a graph $G$ is \defn{$d$-degenerate} if every subgraph of $G$ has minimum degree at most $d$. The minimum such $d$ is the \defn{degeneracy} of $G$. 

\citet{KY03} introduced the following definition. For a graph $G$, total order $\preceq$ of $V(G)$, vertex $v\in V(G)$, and  $s\in\NN$, let \defn{$\reach_s(G,\preceq,v)$} be the set of vertices $w\in V(G)$ for which there is a path $v=w_0,w_1,\dots,w_{s'}=w$ of length $s'\in \{0,\dots,s\}$ such that $w\preceq v$ and $v\prec w_i$ for all $i\in \{0,\dots, s'-1\}$. For a graph $G$ and  $s\in\NN$, the \defn{$s$-strong colouring number} $\col_s(G)$ is the minimum integer $k$ for which there is a total order~$\preceq$ of $V(G)$ with $|\reach_s(G,\preceq,v)|\leq k$ for every vertex $v$ of $G$. Strong colouring numbers interpolate between degeneracy and treewidth~\citep{KPRS16}. Indeed, $\col_1(G)$ equals the degeneracy of $G$ plus 1. At the other extreme, \citet{GKRSS18} showed that $\col_s(G)\leq \tw(G)+1$ for all $s\in\NN$, and indeed $$\lim_{s\to\infty}\col_s(G)=\tw(G)+1.$$

Observe that a graph $H$ is a minor of a graph $G$ if and only if there are pairwise vertex-disjoint subtrees $(T_v)_{v\in V(H)}$ in $G$ such that for each edge $vw\in E(H)$ there is an edge between $T_v$ and $T_w$ in $G$. For $r\in\NN$, if each such tree $T_v$ has radius at most $r$, then $H$ is an \defn{$r$-shallow} minor of $G$. Let $$\nabla_r(G):=\max_H\frac{|E(H)|}{|V(H)|},$$ 
taken over all $r$-shallow (non-empty) minors $H$ of $G$. A graph class $\GG$ has \defn{bounded expansion} if there is a function $f:\mathbb{N}_0\to\mathbb{R}$ such that $\nabla_r(G)\leq f(r)$ for every graph $G\in\GG$ and $r\in\mathbb{N}_0$. A graph class $\GG$ has \defn{linear} or \defn{polynomial expansion} respectively if there is a linear or polynomial expansion function.

\subsection{Product Structure Theorems}
\label{ProductStructure}

For graphs $G$ and $H$, the \defn{strong product} $G\boxtimes H$ is the graph with vertex-set $V(G)\times V(H)$, where vertices $(v,w)$ and $(x,y)$ are adjacent if:
\begin{itemize}
    \item  $v=x$ and $wy\in E(H)$, or 
    \item $w=y$ and $vx \in E(G)$, or
    \item $vx \in E(G)$ and $wy\in E(H)$.
\end{itemize}

The proofs of our main theorems depend on the following recent product structure results. \citet{BDJMW22} defined the \defn{row-treewidth} of a graph $G$ to be the minimum $k\in\NN_0$ such that $G$ is isomorphic to a subgraph of $H\boxtimes P$ for some graph $H$ with treewidth $k$ and path $P$. The motivation for this definition is the following `Planar Graph Product Structure Theorem' of \citet{DJMMUW20} (improved by \citet{UWY22}). 

\begin{thm}[\citep{DJMMUW20,UWY22}]
\label{PlanarProductStructure}
Every planar graph has row-treewidth at most 6.
\end{thm}

\cref{PlanarProductStructure} was generalised for graphs of given Euler genus.

\begin{thm}[\citep{DJMMUW20,UWY22}]
\label{GenusProductStructure}
Every graph of Euler genus $\gamma$ has row-treewidth at most $2\gamma+6$.
\end{thm}

More generally, \citet{DJMMUW20} proved that a minor-closed class has bounded row-treewidth if and only if it excludes some apex graph\footnote{A graph $X$ is \defn{apex} if $X-v$ is planar for some vertex $v$, or $V(X)=\emptyset$.}. For an arbitrary  proper minor-closed class, \citet{DJMMUW20} obtained the following `Graph Minor Product Structure Theorem', where $A+B$ is the \defn{complete join} of graphs $A$ and $B$ (obtained from disjoint copies of $A$ and $B$ by adding every edge with one end-vertex in $A$ and one end-vertex in $B$).

\begin{thm}[\citep{DJMMUW20}] 
\label{MinorFree}
For every graph $X$, there exist $k,a\in\mathbb{N}$ such that every $X$-minor-free graph has a tree-decomposition in which every torso is a subgraph of $(H\boxtimes P)+K_a$ for some graph $H$ of treewidth at most $k$ and some path $P$. 
\end{thm}

Product structure theorems for several non-minor-closed classes are known \citep{DMW23,HW24}. Here is one example.

\begin{thm}[\citep{DMW23}] 
\label{gkPlanarProduct}
Every $(\gamma,k)$-planar graph has row-treewidth $O(\gamma k^6)$.
\end{thm}

With these tools in hand we now give the intuition behind our proofs. First note that \citet{TW-II} showed that $\tww(G\boxtimes H)$ is bounded by a function of $\tww(G)$, $\tww(H)$ and $\Delta(H)$ (and \citet{PS23a} gave analogous results for various other graph products). However, this is not enough to conclude results about subgraphs of $G\boxtimes H$ since twin-width is not monotone. Consider a subgraph $G$ of $H\boxtimes P$ where $H$ has bounded tree-width and $P$ is a path. As mentioned in \cref{Introduction}, $H$ has bounded component twin-width. This says there is a reduction sequence for $H$ such that each red component $X$ has bounded size. Observe that $X\boxtimes P$ has bounded bandwidth. Our strategy is to construct a  reduction sequence for $G$ so that each red component is of the form $X\boxtimes P$ where $X$ is a bounded-size subgraph of $H$, implying each red subgraph has bounded bandwidth. A key to the proof is to find vertex identifications so that the resulting graph stays a subgraph of some $H\boxtimes P$. The above intuitive description has some inaccuracies. Implementing this strategy rigorously needs several further ideas, especially in the setting of powers (\cref{BoundedRowTreewidth}). Finally, to apply \cref{MinorFree} for $K_t$-minor-free graphs we need more ideas to cater for apex vertices and tree-decompositions (\cref{ProperMinorClosedClasses}). 

\section{Distance Profiles}
\label{NeighbourhoodComplexity}

The following concept will help to optimise our bounds on reduced bandwidth and twin-width, and is of independent interest because of connections to VC-dimension~\citep{PP20,GHOORRVS17,RVS19} and sparsity theory~\citep{EGKKPRS17,RVS19}. 

Fix a graph $G$ and $r\in\NN$. 
For vertices $v,w\in V(G)$, let
\[ \dist^r_G(v,w) :=
\begin{cases}
 \dist_G(v,w) & \text{if $\dist_G(v,w)\le r$,}\\
 \infty & \text{otherwise.}
  \end{cases}
\]
For $A\subseteq V(G)$ and $v\in V(G)\setminus A$, the \defn{distance-$r$ profile of $v$ on $A$} is
\[\pi^r_G(v,A):= \{ (w,\dist^r_G(v,w) ) : w\in A \},\]
and 
let
\[\pi^r_G(A) := | \{ \pi^r_G(v,A) : v \in V(G)\setminus A \} |. \]
This definition is different to similar definitions in \citep{EGKKPRS17,RVS19} in that we only consider $v\in V(G)\setminus A$. 
Clearly, 
\begin{equation}\label{equ:profile}
    \pi^r_G(A)\le (r+1)^{|A|}. 
\end{equation}

 Our upper bounds on reduced bandwidth and twin-width are expressed in terms of $\pi^r_G(A)$ for sets $A$ of a given size (see \cref{RowTreewidthNeighbourhoodPowers}). Motivated by this connection, \cref{FirstNeigh} presents various upper bounds on $\pi^1_G(A)$ that are tight  for graphs of given Euler genus (\cref{SurfaceNeighbours}) and $K_t$-minor-free graphs (\cref{MinorfreeNeighbour}). \cref{SecondNei} gives an upper bound on $\pi^2_G(A)$ where $G$ is a graph of given Euler genus. 

\subsection{First Neighbourhoods} 
\label{FirstNeigh}

This subsection presents bounds on $\pi^1_G(A)$ for various graphs $G$. Note that 
 \[\pi^1_G(A)=|\{N_G(u)\cap A:u\in V(G)\setminus A\}|.\] 
So in a monotone class, we may assume that $G$ is bipartite with bipartition $\{A,V(G)\setminus A\}$, where $N_G(u)\cap X\neq N_G(v)\cap X$ for distinct $u,v\in V(G)\setminus A$. 

The following lemma is a more precise version of a result by \citet{GHOORRVS17} (see \cref{NeighNabla} below). For a graph $H$ and $k\in\NN_0$, let \defn{$C(H,k)$} be the number of cliques of order $k$ in $H$, where $\emptyset$ is considered to be the only clique of order $0$ in $G$. So $C(H,0)=1$, $C(H,1)=|V(H)|$, and $C(H,2)=|E(H)|$. Let \defn{$C(H,\leq k)$} be the number of cliques of order at most $k$ in $H$, and let \defn{$C(H)$} be the total number of cliques in $H$. 

\begin{lem}
\label{Neighbourlyness}
Let $G$ be a bipartite graph with bipartition $\{X,Y\}$, where $K_t$ is not a 1-shallow minor of $G$. Then there is a 1-shallow minor $H$ of $G$ on $|X|$ vertices, such that 
$$|\{N_G(u):u\in Y\}| \leq 
\begin{cases}
C(H,\leq t-2) & \text{if }t\geq 4\\
C(H,\leq 2) & \text{if }t=3.\\
\end{cases}$$
\end{lem}

\begin{proof}
We may assume that $N_G(u)\neq N_G(v)$ for all distinct $u,v\in Y$. 
For $i\in\NN_0$, let $Y_i:=\{v\in Y: \deg_G(v)=i\}$. 
For each $v\in Y_2$, let $X_v:=N_G(v)$. 
By assumption, $X_v\neq X_w$ for distinct $v,w\in Y_2$. 
Let $A$ be a maximal set such that:
\begin{itemize}
    \item $Y_2\subseteq A\subseteq Y\setminus(Y_0\cup Y_1)$, and 
    \item for each $v\in A$ there exists $X_v\subseteq N_G(v)$ with $|X_v|=2$ where $X_v\neq X_w$ for all distinct $v,w\in A$.
\end{itemize} 

Let $H$ be the graph obtained from $G[X\cup A]$, where for each $v\in A$, we pick one $w\in X_v$ and contract the edge $vw$. So $H$ is a 1-shallow minor of $G$ (with each branch set centred at a~vertex in $X$), where $|X|=|V(H)|$ and $|A|=|E(H)|=C(H,2)$. 

Let $B:=Y\setminus(Y_0\cup Y_1\cup A)$. Consider a vertex $w\in B$ with $d:= \deg_G(w)$. So $d\geq 3$ since $Y_2\subseteq A$. By the maximality of $A$, we have $N_G(w)$ is a clique in $H$, implying $K_{d+1}$ is a 1-shallow minor of $G$. Thus $d\in\{3,4,\dots,t-2\}$. 
For distinct $v,w\in B$, we have $N_G(v)\neq N_G(w)$. So $|B|\leq \sum_{i=3}^{t-2} C(H,i)$, and $B=\emptyset$ if $t\leq 4$. 

Also $|Y_0|\leq 1=C(H,0)$ and $|Y_1|\leq|X|=C(H,1)$. 
Therefore $|Y|=|Y_0|+|Y_1|+|A|+|B|\leq C(H,0)+C(H,1)+C(H,2)+|B|$. If $t\leq 4$ then $B=\emptyset$ and 
$|Y|\leq C(H,\leq 2)$; otherwise $|Y|\leq C(H,0)+C(H,1)+C(H,2)+
\sum_{i=3}^{t-2} C(H,i) = C(H,\leq t-2)$.
\end{proof}

It is well known (see \citep{Wood07,Wood16}) that every $d$-degenerate graph $G$ with $n\geq d$ vertices satisfies:
\begin{itemize}
    \item $C(G,k)\leq \binom{d}{k-1}(n-\frac{(k-1)(d+1)}{k})$ for all $k\in\{0,1,\dots,d+1\}$,
    \item $C(G)\leq 2^d(n-d+1)$.
\end{itemize}
So \cref{Neighbourlyness} implies:

\begin{cor}
\label{NeighDegen}
For every graph $G$ and set $A\subseteq V(G)$, 
if every 1-shallow minor of $G$ on $|A|$ vertices is $d$-degenerate
and $d\leq |A|$, then 
$\pi^1_G(A) \leq 2^d(|A|-d+1)$.
\end{cor}

\cref{NeighDegen} is applicable with $d=\floor{2\nabla_1(G) }$ since every 1-shallow minor of a graph $G$ is $\floor{2\nabla_1(G)}$-degenerate. We obtain the following lemma, which is a slight strengthening  of a result of  \citet[Lemma~4.3]{GHOORRVS17}. 

\begin{lem}\label{NeighNabla}
For every graph $G$ with $\nabla_1(G) \leq k$ and for every set $A\subseteq V(G)$  
with $|A|\geq \floor{2k}$, 
$$\pi^1_G(A) \leq 
 2^{\floor{2k}}(|A|-\floor{2k}+1)  .$$
\end{lem}

\begin{lem}
\label{NeighCol}
For every graph $G$ with $\col_5(G)\leq c$ and for every set $A\subseteq V(G)$ with $|A| \geq c-1$, 
\[ \pi^1_G(A) \leq  2^{c-1}( |A|-c+2).\]
\end{lem}

\begin{proof}
\citet[Lemma~19]{HW24} proved that for every graph $G$ and every $r$-shallow minor $G'$ of $G$, we have 
$\col_s(G') \leq \col_{2rs+2r+s}(G)$. 
Every graph $G$ is $(\col_1(G)-1)$-degenerate.
Thus every 1-shallow minor of a graph $G$ is $(\col_5(G)-1)$-degenerate. The result follows from \cref{NeighDegen}.
\end{proof}

\begin{lem}\label{NeighGammak}
For every $(\gamma,k)$-planar graph $G$ and for every set $A\subseteq V(G)$
with $|A| \geq 22(2\gamma+3)(k+1)-1$, 
\[ \pi^1_G(A) \leq 2^{22(2\gamma+3)(k+1)-1}(|A|-22(2\gamma+3)(k+1)+2) .\]
\end{lem}

\begin{proof}
Van den Heuvel and Wood~\citep{vdHW18} proved that $\col_s(G)\leq 2(2\gamma+3)(k+1)(2s+1)$.
The result follows from \cref{NeighCol}. 
\end{proof}

\begin{lem}
\label{NeighRowTreewidth}
For every graph $G$ with row-treewidth at most $k$ and for every set $A\subseteq V(G)$ with $|A|\geq 11k+10$, 
$$\pi^1_G(A) \leq 2^{11k+10}(|A|-11k-9).$$
\end{lem}

\begin{proof}
By assumption, $G$ is isomorphic to a subgraph of $H\boxtimes P$, where $\tw(H)\leq k$ and $P$ is a path. \citet[Lemma~21]{HW24} showed that $$\col_s(H \boxtimes P) \leq (2s+1) \col_s(H) \leq (2s+1) (\tw(H)+1 ).$$ 
In particular, $\col_5(G) \leq \col_5(H \boxtimes P) \leq 11(k+1)$. The result follows from \cref{NeighCol}.
\end{proof}

We get improved bounds for minor-closed classes. For every $n$-vertex $K_t$-minor-free graph $G$, \citet{Kostochka84} and \citet{Thomason84} independently showed that $G$ has  $O(t\sqrt{\log t}\, n)$ edges, and \citet{FW17} showed that $G$ has at most $3^{2t/3+o(t)}n$ cliques. \cref{Neighbourlyness} implies:

\begin{cor}\label{MinorfreeNeighbour}
For every $K_t$-minor-free graph $G$ and for every set $A\subseteq V(G)$,
$$\pi^1_G(A) \leq 3^{2t/3+o(t)}|A|+1.$$
\end{cor}

We now show that this bound is tight up to the $o(t)$ term. We may assume $2t=3p+2$ for some even $p\in\NN$. Let $H$ be the complete $p$-partite graph $K_{2,\dots,2}$, which can be obtained from $K_{2p}$ by deleting a perfect matching $v_1w_1,\dots,v_pw_p$. Obviously, the largest complete subgraph in $H$ is $K_p$. \citet{Wood07} showed that the largest complete graph minor in $H$ is $K_{3p/2}$ (with branch sets $\{v_1\},\dots,\{v_p\},\{w_1,w_2\},\dots,\{w_{p-1},w_p\}$). So  $H$ is $K_t$-minor-free (since $t>\frac32 p$). It is well known that $H$ has exactly $3^p$ cliques (since if $C_i$ is any element of $\{\emptyset,\{v_i\},\{w_i\}\}$, then $C_1\cup\dots\cup C_p$ is a clique, giving $3^p$ cliques in total, and every clique in $H$ is obtained this way). Let $G$ be the bipartite graph with bipartition $\{X,Y\}$, where $X=V(H)$ and for each clique $C$ in $H$, there is a one vertex $y_C$ in $Y$ with $N_G(y_C)=C$. So $G$ can be obtained from $H$ by clique-sums with complete graphs of order at most $p+1$ (and edge-deletions). So $G$ is also $K_t$-minor-free, and every vertex in $Y$ has a unique neighbourhood. Thus  $\pi^1_G(X)=|Y|=3^p=3^{(2t-2)/3}=3^{2t/3-O(\log t)}|X|$. 

\begin{cor}
\label{NeighbourlynessTreewidth}
For every graph $G$ with treewidth $k\in\NN$ and for every set $A\subseteq V(G)$, 
\[ \pi^1_G(A) \leq 
\begin{cases}
2|A| & \text{if }k=1\\
2^{|A|} & \text{if }k\geq 2\text{ and }|A|\leq k\\
(2^k-1)(|A|-k)+2^k & \text{if }k\geq 2\text{ and }|A|\geq k\\
\end{cases}\]
\end{cor}

\begin{proof}
Let $Y:=V(G)\setminus A$. First suppose that $k=1$. So $G$ and every minor of $G$ are forests. So \cref{Neighbourlyness} is applicable with $t=3$. Thus, there is a minor $H$ of $G$ on $|A|$ vertices, such that 
$\pi^1_G(A)=|\{N_G(u)\cap A : u\in Y\}| \leq C(H)\leq 2|A|$ since $H$ is a forest.  

Now assume that $k\geq 2$. The class of treewidth-$k$ graphs is proper minor-closed and has no $K_{k+2}$ minor. So  \cref{Neighbourlyness} is applicable with $t=k+2\geq 4$. Thus there is a minor $H$ of $G$ on $|A|$ vertices, such that $|\{N_G(u):u\in Y\}| \leq C(H,\leq k)$. This is at most $2^{|A|}$ if $|A|\leq k$. Now assume that $|A|\geq k+1$. Every graph with treewidth at most $k$ is $k$-degenerate. So  $H$ is $k$-degenerate, implying 
\[\pi^1_G(A)= |\{N_G(u):u\in Y\}| \leq \sum_{i=0}^{k} \tbinom{k}{i-1} \big(|A|-\tfrac{(i-1)(k+1)}{i}\big) = 2^k(|A|-k+1) -|A|+k.\qedhere\]
\end{proof}

The bound in \cref{NeighbourlynessTreewidth} is tight: Let $H$ be a $k$-tree on $n\geq k$ vertices, which has $2^k(n-k+1)$ cliques and $n-k$ cliques of size $k+1$~\citep{Wood16}. Let $G$ be the bipartite graph with bipartition $\{X,Y\}$, where $X=V(H)$ and for each clique $C$ with $|C|\leq k$ in $H$, there is a one vertex $y_C$ in~$Y$ with $N_G(y_C)=C$. So $|Y|=C(H,\leq k)=2^{k}(n-k+1)-n+k$. Given a tree-decomposition of $H$ with width $k$, for each clique $C$ in $H$ with $|C|\leq k$, which must be in some bag $B$, add one new bag $B'=B\cup\{y_C\}$ adjacent to $B$ to obtain a tree-decomposition of $G$ with width $k$.

We have the following bound for graphs of given Euler genus. It follows from Euler's formula that $|E(G)|\leq 3(|V(G)|+\gamma-2)$ for every graph $G$ with Euler genus $\gamma$; moreover, $|E(G)|\leq 2(|V(G)|+\gamma-2)$ if $G$ is bipartite. 

\begin{lem}
\label{SurfaceNeighbours}
For every graph $G$ with Euler genus $\gamma$ and for every set $X\subseteq V(G)$ with $|X|\geq 2$,
\[\pi^1_G(X) \leq  6|X|+5\gamma -9.\]
\end{lem}

\begin{proof}
We may assume that $G$ is bipartite with bipartition $\{X,Y\}$, and that $N_G(u)\neq N_G(v)$ for distinct $u,v\in Y$ (otherwise delete one of the vertices). 
For $i\in\mathbb{N}_0$, let $Y_i:= \{u\in Y: \deg_G(u)=i\}$. 
Let $G':=G-Y_0-Y_1$. 
Since $G'$ is bipartite, 
\begin{equation*}
|E(G)|-|Y_1| 
= |E(G')|\leq 2(|V(G')|+\gamma-2) =
2( |V(G)|-|Y_0|-|Y_1|+\gamma-2).
\end{equation*}
Thus
\begin{equation*}
\sum_{i\geq 0} i |Y_i| 
= |E(G)| \leq 2 |V(G)|- 2|Y_0| - |Y_1| +2\gamma-4 
= 2\Big( |X|+\sum_{i\geq 0} |Y_i| \Big) - 2|Y_0| - |Y_1| +2\gamma-4.
\end{equation*}
Hence
\begin{equation*}
|Y| - |Y_0| - |Y_1| - |Y_2| = 
\sum_{i\geq 3} |Y_i| \leq 
\sum_{i\geq 3} (i-2) |Y_i| \leq 
2 |X| +2\gamma - 4,
\end{equation*}
implying $|Y| \leq  2 |X| + |Y_0| + |Y_1| +|Y_2| +2\gamma- 4$. Since $N_G(u)\neq N_G(v)$ for distinct $u,v\in Y$, we have $|Y_0|\leq 1$ and $|Y_1|\leq|X|$. If $H$ is obtained from $G[X\cup Y_2]$ by contracting one edge incident to each vertex in $Y_2$, then $H$ has no parallel edges and $|X|$ vertices, implying $|Y_2|=|E(H)|\leq 3(|X|+\gamma-2)$. Hence \begin{equation*}
    |Y|  \leq 2|X| + 1 + |X| +  3(|X|+\gamma-2) +2\gamma - 4 = 6|X|+5\gamma-9.\qedhere
\end{equation*}
\end{proof}

\cref{SurfaceNeighbours} is also tight: Let $G_0$ be a triangulation of a surface with Euler genus $\gamma$ with at least four vertices. Let $G$ be obtained from $G_0$ as follows: add one vertex adjacent to the three vertices of each face of $G_0$, subdivide each edge of $G_0$, add one vertex adjacent to each vertex of $G_0$, and add one isolated vertex. So $G$ is bipartite with bipartition $\{X,Y\}$ where $X:=V(G_0)$ and $Y:=V(G)\setminus X$. No two vertices in $Y$ have the same neighbourhood, and $|Y|=|Y_0|+|Y_1|+|Y_2|+|Y_3|=1+|X|+3(|X|+\gamma-2)+2(|X|+\gamma-2)=6|X|+5\gamma-9$.

\subsection{Second Neighbourhoods}
\label{SecondNei}

This section gives bounds on the distance-$2$ profiles in graphs of given Euler genus. These results are useful for bounding the reduced bandwidth of squares and map graphs. 

\begin{lem}\label{SurfaceSecondNeighbourhood} 
Let $G$ be a graph of Euler genus $\gamma$, and let $X\subseteq V(G)$ with $|X|\geq 2$. Let $Y:=N_G(X)$ and $Z:=V(G)\setminus (X\cup Y)$. Then $$| \{ N^2_G(v)\cap X : v\in Z \}| \leq  
(60 \gamma^2 + 125 \gamma + 68)|X|-120 \gamma^2 -250 \gamma -132.$$
\end{lem}

\begin{proof}
We may assume that $N^2_G(v)\cap X\neq N^2_G(w)\cap X$ for all distinct $v,w\in Z$. Let $Z^0:=\{v\in Z: |N^2_G(v)\cap X|=0\}$ and $Z^1:=\{v\in Z: |N^2_G(v)\cap X|=1\}$ and $Z^2:=\{v\in Z: |N^2_G(v)\cap X|\geq 2\}$.  
Thus $|Z^0|\leq 1$ and $|Z^1|\leq|X|$. 
By \cref{SurfaceSecondNeighbourhoodInduction} below, 
$|Z^2|\leq (60 \gamma^2 + 125 \gamma + 67)(|X|-2)+1$. 
In total, $|Z|\leq 
(60 \gamma^2 + 125 \gamma + 67)(|X|-2)+|X|+2
= 
(60 \gamma^2 + 125 \gamma + 68)|X|-120 \gamma^2 -250 \gamma -132$, as desired.
\end{proof}

The proof of \cref{SurfaceSecondNeighbourhoodInduction} uses the next lemma, which follows from  \citep[Proposition~4.2.7]{MoharThom} and the discussion after it.

\begin{lem}
\label{FindContractiblecycle}
If $K_{2,2\gamma+2}$ is embedded in a surface of Euler genus $\gamma$, then some 4-cycle in $K_{2,2\gamma+2}$ is contractible.
\end{lem}

\begin{lem}
\label{SurfaceSecondNeighbourhoodInduction}
Let $G$ be a graph of Euler genus $\gamma$, and let $X\subseteq V(G)$ with $|X|\geq 2$. Let $Y:=N_G(X)$ and $Z:=V(G)\setminus(X\cup Y)$. 
Assume that $|N^2_G(v)\cap X| \geq 2$ for every vertex $v\in Z$, and that $N^2_G(v)\cap X\neq N^2_G(w)\cap X$ for all distinct $v,w\in Z$. 
Then $|Z|\leq (60 \gamma^2 + 125 \gamma + 67)(|X|-2)+1$. 
\end{lem}

\begin{proof}
We prove that $|Z|\leq c(|X|-2)+1$ by induction on $|Y|$, where $c:=60 \gamma^2 + 125 \gamma + 67$. (This choice of $c$ will become clear at the end of the proof.)\ In the base case, if $|X|=2$ then $|Z|\leq 1=c(|X|-2)+1$. Now assume that $|X|\geq 3$. We may assume that each of $X$, $Y$ and $Z$ are independent sets. Consider an embedding of $G$ into a surface of Euler genus $\gamma$. 

First suppose that there is a set $Y_0\subseteq Y$ and distinct vertices $x_1,x_2\in X$ such that $|Y_0|\geq 2\gamma+2$ and $N_G(y)\cap X=\{x_1,x_2\}$ for each $y\in Y_0$. Thus $G[\{x_1,x_2\}\cup Y_0]$ contains a $K_{2,2\gamma+2}$ subgraph. By \cref{FindContractiblecycle}, there is a contractible 4-cycle $C=(x_1,y_1,x_2,y_2)$ in $G$, for some $y_1,y_2\in Y_0$. So $C$ bounds a disc $D$. Let $G_1$ be the subgraph of $G$ induced by the vertices embedded in $D$, and let $G_2$ be the subgraph of $G$ induced by the vertices embedded in the boundary of $D$ or not in $D$. Thus $G=G_1\cup G_2$ and $G_1\cap G_2=C$, where $C$ is the boundary of a face of both $G_1$ and $G_2$. 

Let $G'$ be the graph obtained from $G_1$ by identifying $y_1$ and $y_2$ into a vertex $y'$. Let $G''$ be the graph obtained from $G_2$ by identifying $y_1$ and $y_2$ into a vertex $y''$. Since $y_1$ and $y_2$ are on a common face before their identification, $G'$ and $G''$ have Euler genus at most $\gamma$. 

Let $X':=V(G')\cap X$ and $X'':=V(G'')\cap X$. 
Note that $X'\cap X''=\{x_1,x_2\}$.
Let $Y':=(V(G')\cap Y)\cup\{y'\}$ and 
$Y'':=(V(G'')\cap Y)\cup\{y''\}$. 
Note that $|Y'|<|Y|$ and $|Y''|<|Y|$. 
Let $Z':=V(G')\cap Z$ and $Z'':=V(G'')\cap Z$. 
Note that $Z'$ and $Z''$ partition $Z$. 

We claim that $|N^2_{G'}(v)\cap X'|\geq 2$ for each $v\in Z'$. Suppose that $|N^2_{G'}(v)\cap X'|\leq 1$. Then there is a vertex from $N^2_G(v)\cap X$ in $G''-V(G')$. Since $C$ is separating, $v$ is adjacent to $y_1$ or $y_2$ in~$G$, implying $x_1,x_2\in N^2_{G'}(v)$ and $|N^2_{G'}(v)\cap X'|\geq 2$, as desired. Similarly, $|N^2_{G''}(v)\cap X''|\geq 2$ for each $v\in Z''$. 

Since $N_{G'}(y')\cap X'=N_G(y_1)\cap X=N_G(y_2)\cap X=\{x_1,x_2\}$, we have $N^2_{G'}(v)\cap X' = N^2_G(v)\cap X$ for every $v\in Z'$. Hence $N^2_{G'}(v)\cap X'\neq N^2_{G'}(w)\cap X'$ for distinct vertices $v,w\in Z'$. Similarly, $N^2_{G''}(v)\cap X''\neq N^2_{G''}(w)\cap X''$ for distinct vertices $v,w\in Z''$. 

By assumption, there is at most one vertex $v\in Z$ with  $N^2_G(v)\cap X=N_G(y_1)\cap X$. Without loss of generality, if such a vertex $v$ exists, then $v$ is not in $Z''$. Add a new vertex $z$ to $Z''$ and to $G''$ only adjacent to $y''$ in $G''$. Then $x_1,x_2\in N^2_{G''}(z)\cap X''$ and $|N^2_{G''}(z)\cap X''|\geq 2$ as required. If $N^2_{G''}(z)\cap X'' = N^2_{G''}(v)\cap X''$ for some vertex $v\in Z''\setminus\{z\}$, then $N^2_{G''}(v)\cap X'' = N^2_{G''}(z)\cap X''=N_G(y_1)$, which contradicts the above property of $Z''$. Hence $N^2_{G''}(z)\cap X'' \neq N^2_{G''}(v)\cap X''$ for every vertex $v\in Z''\setminus\{z\}$.

Now $|Z_1|+|Z_2|=|Z|+1$. We have shown that $X'$, $Y'$ and $Z'$ satisfy the assumptions of the inductive hypothesis within $G'$. Since $|Y'|<|Y|$, by induction, $|Z'|\leq c(|X'|-2)+1$. Similarly, $|Z''|\leq c(|X''|-2)+1$. Hence
\begin{align*}
|Z|\,=\,|Z'|+|Z''|-1  & \leq c(|X'|-2)+1 +  c(|X''|-2)+1-1  \\
&= c( |X'|+|X''|-4)+1 \\
& = c(|X|-2)+1,
\end{align*}
as desired. Now assume that there are no such vertices $x_1,x_2\in X$ and set $Y_0$. 

As illustrated in \cref{Hgraph}, let $Y^1$ be the set of vertices in $Y$ with exactly one neighbour in $X$. Let $Y^1_1,\dots,Y^1_p$ be the partition of $Y^1$, where for all $v\in Y^1_i$ and $w\in Y^1_j$ we have $N_G(v)\cap X=N_G(w)\cap X$ if and only if $i=j$. Let $y^1_i$ be a vertex in $Y^1_i$ and let $x_i$ be the neighbour of $y^1_i$ in $X$. 

Let $Y^2$ be the set of vertices in $Y$ with exactly two neighbours in $X$. 
Let $Y^2_1,\dots,Y^2_q$ be the partition of $Y^2$, where for all $v\in Y^2_i$ and $w\in Y^2_j$ we have $N_G(v)\cap X=N_G(w)\cap X$ if and only if $i=j$. As shown above, $|Y^2_i|\leq 2\gamma+1$ for each $i\in\{1,\dots,q\}$. Let $y^2_i$ be a vertex in $Y^2_i$. 

Let $Y^3$ be the set of vertices in $Y$ with at least three neighbours in $X$. Let $Y^3_1,\dots,Y^3_r$ be the partition of $Y^3$, where for all $v\in Y^3_i$ and $w\in Y^3_j$ we have $N_G(v)\cap X=N_G(w)\cap X$ if and only if $i=j$. Since $K_{3,2\gamma+3}$ has Euler genus greater than $\gamma$, $|Y^3_i|\leq 2\gamma +2$ and $|Y^3|\leq (2\gamma+2)r$. Let $y^3_i$ be a vertex in $Y^3_i$. 

By construction, the vertices $y^1_1,\dots,y^1_p,y^2_1,\dots,y^2_q,y^3_1,\dots,y^3_r$ have pairwise distinct non-empty neighbourhoods in $X$. By \cref{SurfaceNeighbours} applied to the bipartite graph between $X$ and $\{y^1_1,\dots,y^1_p,y^2_1,\dots,y^2_q,y^3_1,\dots,y^3_r\}$, we have $p+q+r\leq (6|X|+5\gamma-9)-1=6|X|+5\gamma-10$, because $y^1_1,\dots,y^1_p,y^2_1,\dots,y^2_q,y^3_1,\dots,y^3_r$ have non-empty neighbours in $X$. In fact, $p\leq |X|$ and $q+r\leq 5|X|+5\gamma-10$.

\begin{figure}
\centering
\includegraphics{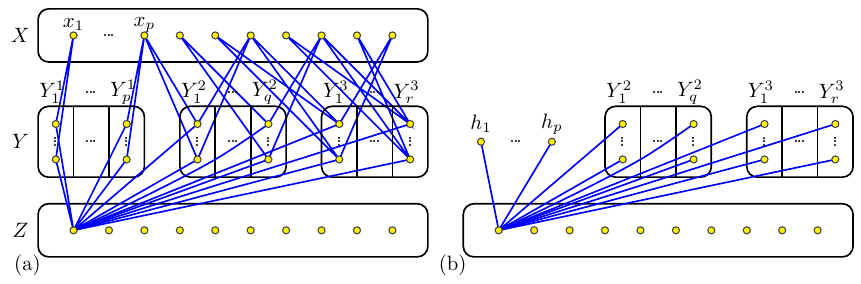}
\caption{(a) The sets $X,Y^1_1,\dots,Y^1_p,Y^2_1,\dots,Y^2_q,Y^3_1,\dots,Y^3_r,Z$ in $G$. (b) The graph $H$.}
\label{Hgraph}
\end{figure}

Let $H$ be the graph obtained from $G$ as follows: 
delete $X\setminus\{x_1,\dots,x_p\}$, 
delete any edges between $\{x_1,\dots,x_p\}$ and $Y^2\cup Y^3$, and 
for each $i\in\{1,\dots,p\}$ contract $\{x_i\}\cup Y^1_i$ (which induces a star) into a new vertex $h_i$. Note that $H$ is bipartite and planar, with one colour class $Z$ and the other colour class $\{h_1,\dots,h_p\}\cup Y^2\cup Y^3$, which has size at most 
$$p+|Y^2|+|Y^3| 
\leq p + (2\gamma+1)q + (2\gamma+2)r 
\leq  p + (2\gamma+2)(q+r)
\leq  |X| + (2\gamma+2)(5|X|+5\gamma-10).$$ 
For each $v\in Z$, 
$$N^2_G(v)\cap X = 
\{ x_i : h_i\in N_H(v) \} \cup 
\Big( \bigcup_{u\in N_H(v)\setminus\{h_1,\dots,h_p\}} \!\!\!\!\!\!\!\!\!\!\!\! N_G(u)\cap X \Big).$$
Thus $N^2_G(v)\cap X$ is determined by $N_H(v)$. 
For all distinct $v,w\in Z$, since
$N^2_G(v)\cap X\neq N^2_G(w)\cap X$, we have $N_H(v)\neq N_H(w)$.
By \cref{SurfaceNeighbours} applied to $H$, 
\begin{align*}
    |Z|
&\leq 6 (p+|Y^2|+|Y^3|) +5\gamma - 10 \\
&\leq 6\big( |X| + (2\gamma+2)(5|X|+5\gamma-10) \big) +5\gamma - 10  \\
& \leq c(|X|-2)+1,
\end{align*} 
since $|X|\geq 3$. Indeed, $c$ is defined so that this final inequality holds. 
\end{proof}

\begin{lem}
\label{Nu2}
Let $f:\NN\to\NN$ be a function and let $\GG$ be a monotone class, such that for every $G\in\GG$ and $X\subseteq V(G)$ and $Z\subseteq V(G)\setminus N_G[X]$,
\[| \{ N^2_G(v)\cap X : v\in Z \}| \leq f(|X|).\]
Then $\pi^2_G(X)\leq \pi^1_G(X)\, f(|X|)$. 
\end{lem}

\begin{proof}
Let $s:= \pi^1_G(X)$ and $t:= f(|X|)$. 
Let $Y:=V(G)\setminus X$. 
Let $Y_1,\dots,Y_s$ be a partition of $Y$ where $v,w\in Y_i$ if and only if $N_G(v)\cap X=N_G(w)\cap X$.
For each $i\in \{1, \ldots, s\}$, let $G_i$ be the graph obtained from $G$ by deleting the edges between $Y_i$ and $X$. 
Since $\GG$ is monotone, $G_i\in\GG$, implying
$|\{ N^2_{G_i}(v) \cap X : v \in Y_i \}| \leq t$.
Let $Y_{i,1},\dots,Y_{i,t}$ be a partition of $Y_i$ where
$v,w\in Y_{i,j}$ if and only if $N^2_{G_i}(v)\cap X = N^2_{G_i}(w)\cap X$. It follows that for each $i\in \{1, \ldots, s\}$ and $j\in\{1, \ldots, t\}$, 
we have $\pi_2^G(v,X) = \pi_2^G(w,X)$ for all $v,w\in Y_{i,j}$. Thus $\pi^2_G(X)\leq st$, as desired. 
\end{proof}

\cref{SurfaceNeighbours,SurfaceSecondNeighbourhood,Nu2} imply the following bound on $\pi^2_G$ for graphs of given Euler genus. 

\begin{cor}\label{SurfaceNu2}  
For every graph $G$ of Euler genus $\gamma$ and for every set $X\subseteq V(G)$ with $|X|\geq 2$, 
\[ \pi^2_G(X) \leq (6 |X| + 5\gamma -9)
( (60 \gamma^2 + 125 \gamma + 68)|X|-120 \gamma^2 -250 \gamma -132). \]
\end{cor}

\section{Bounded Row-Treewidth Classes}
\label{BoundedRowTreewidth}
\label{PlanarGraphs}

This section presents upper bounds on reduced bandwidth and twin-width for powers of graphs with bounded row-treewidth, which includes planar graphs, graphs with Euler genus $\gamma$, $(\gamma,k)$-planar graphs, and map graphs. The heart of the proof is \cref{ProductPath} below, which depends on the following definition. As illustrated in \cref{fig:sxq}, for $x,q\in\NN$ with $q\geq 2$, let \defn{$S^*_{x,q}$} be the graph where:
\begin{itemize}
    \item $V(S^*_{x,q})$ is the disjoint union of sets $Q,A_1,\dots,A_x,B_1,\dots,B_x,C_1,\dots,C_x$ with $|Q|=2q-1$ and $|A_i|=|B_i|=|C_i|=q$ for all $i\in\{1,\dots,x\}$, and
    \item $Q\cup A_1$ is a clique, 
$Q\cup B_1$ is a clique, 
$Q\cup C_1$ is a clique, and for all $\in\{1,\dots,x-1\}$, 
$A_i\cup A_{i+1}$ is a clique, 
$B_i\cup B_{i+1}$ is a clique, and
$C_i\cup C_{i+1}$ is a clique, and 
\item any two vertices that do not belong to the same clique among the cliques above are non-adjacent.
\end{itemize}  
For $x,q,r\in\NN$ with $q\geq 2$, let $S_{x,q,r}$ be the graph obtained from $(S^*_{x,q})^r$ by removing all edges between $B_1\cup\dots\cup B_x$ and $C_1\cup\dots\cup C_x$. We call $Q$ the \defn{center} of $S_{x,q,r}$. See Figure~\ref{fig:s423} for an illustration of $S_{4,2,3}$. Note that the vertices in $Q$ have maximum degree in $S_{x,q,r}$. Thus 
\begin{equation}
\label{DeltaS}
\Delta(S_{x,q,r})\le (3r+2)q-2.
\end{equation}
\begin{figure}[!ht]
\centerline{\includegraphics[scale=0.3]{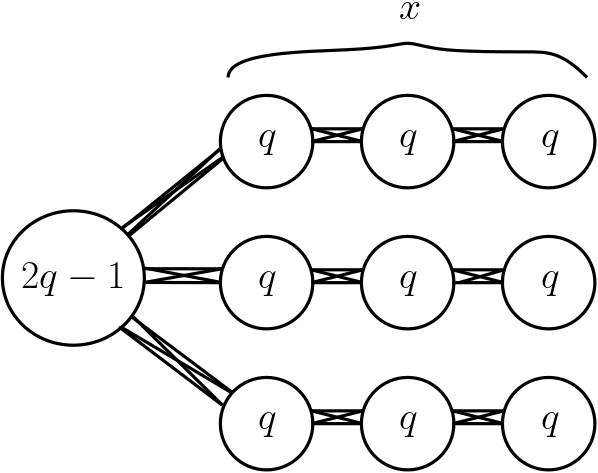} }
\caption{An illustration of $S_{x,q}^*$.}
\label{fig:sxq}
\end{figure}

\begin{figure}
    \centering
    \begin{tikzpicture}[scale=0.8]
        \tikzstyle{v}=[circle, draw, solid, fill=black, inner sep=0pt, minimum width=3pt]
        \tikzstyle{w}=[circle, draw, thick, inner sep=8pt]

         \node[v] (a1) at (0,.2) {};
        \node[v] (a2) at (-.2,-.2) {};
        \node[v] (a3) at (.2,-.2) {};
         \node[w] (A) at (0, 0) {};

        \foreach \x in {1,2,3,4} 
        {
         \node[w] (U\x) at (-\x * 2, 0) {};
         \node[w] (V\x) at (\x * 2, .7) {};
         \node[w] (W\x) at (\x * 2, -.7) {};
         \node[v] (u\x1) at (-\x * 2,.2) {};
         \node[v] (u\x2) at (-\x * 2,-.2) {};
         \node[v] (v\x1) at (\x * 2,.5) {};
         \node[v] (v\x2) at (\x * 2,.9) {};
         \node[v] (w\x1) at (\x * 2,-.5) {};
         \node[v] (w\x2) at (\x * 2,-.9) {};
         
        \draw[thick] (v\x1)--(v\x2);
        \draw[thick] (w\x1)--(w\x2);
        \draw[thick] (u\x1)--(u\x2);
        }
    \draw[thick] (a1)--(a2)--(a3)--(a1);

        \draw[very thick] (A)--(U1);
        \draw[very thick] (U1)--(U2);
        \draw[very thick] (U2)--(U3);
        \draw[very thick] (U3)--(U4);
       \draw[very thick] (A)--(W1);
        \draw[very thick] (W1)--(W2);
        \draw[very thick] (W2)--(W3);
        \draw[very thick] (W3)--(W4);
       \draw[very thick] (A)--(V1);
        \draw[very thick] (V1)--(V2);
        \draw[very thick] (V2)--(V3);
        \draw[very thick] (V3)--(V4);
        
        \draw[bend left=30, very thick] (U4) to (U2);
        \draw[bend left=30, very thick] (U3) to (U1);
        \draw[bend left=30, very thick] (U2) to (A);
        \draw[bend left=30, very thick] (U1) to (V1);
        \draw[bend left=50, very thick] (A) to (V2);
        \draw[bend left=30, very thick] (V1) to (V3);
        \draw[bend left=30, very thick] (V2) to (V4);

        \draw[bend right=30, very thick] (U1) to (W1);
        \draw[bend right=50, very thick] (A) to (W2);
        \draw[bend right=30, very thick] (W1) to (W3);
        \draw[bend right=30, very thick] (W2) to (W4);        
        
        \draw[bend left=40, very thick] (U4) to (U1);
        \draw[bend left=40, very thick] (U3) to (A);
        \draw[bend left=40, very thick] (U2) to (V1);
        \draw[bend left=40, very thick] (U1) to (V2);
        \draw[bend left=60, very thick] (A) to (V3);
        \draw[bend left=40, very thick] (V1) to (V4);

        \draw[bend right=40, very thick] (U2) to (W1);
        \draw[bend right=40, very thick] (U1) to (W2);
        \draw[bend right=60, very thick] (A) to (W3);
        \draw[bend right=40, very thick] (W1) to (W4);
    \end{tikzpicture}
    \caption{The graph $S_{4,2,3}$.}
    \label{fig:s423}
\end{figure}
Considering the vertex-ordering $A_x,A_{x-1},\dots,A_1,Q,B_1,C_1,B_2,C_2,\dots,B_x,C_x$ we see that 
\begin{equation}
\label{bwS}
\bw(S_{x,q,r})\leq |Q|+\sum_{i=1}^r(|B_i|+|C_i|)-1=(2q-1)+2qr-1=(2r+2)q-2.
\end{equation}

\begin{lem}\label{ProductPath}
Let $k, q, r\in\NN$ with $q\ge k+1$. Let $f$ be a monotone and union-closed graph parameter and let $g:\mathbb{N}\times \mathbb{N}\to\mathbb{R}$ be a function such that $f(S_{x,q,r})\le g(q,r)$ for all $x\in \NN$.
Let $P$ be a path and let $H$ be a graph admitting a $(k,q)$-rooted tree-decomposition $(T, \cB)$. Let $F$ be a trigraph with  $V(F)\subseteq V(H\boxtimes P)$ (not necessarily a subgraph of $H\boxtimes P$) such that: 
 \begin{itemize}
    \item \textup{(\defn{red edge condition})} for every red edge $vw$ of $F$, there is a leaf bag $B$ with parent $B'$ in $(T, \cB)$ such that $v,w\in (B\setminus B')\times V(P)$;
     \item \textup{(\defn{separation condition})} for every rooted separation $(C, D)$ of $H$ from $(T, \cB)$ and every $z\in V(P)$, we have
 $\left| \left\{N_{F}(v)\cap (D\times V(P)):v\in ((C\setminus D)\times \{z\})\cap V(F)\right\}  \right| \le q$; and
   \item \textup{(\defn{neighbourhood condition})} for every $z\in V(P)$ and $v\in (V(H)\times \{z\}) \cap V(F)$, we have $N_{F}[v]\subseteq V(H)\times N^r_P[z]$.
 \end{itemize}
Then $\reduced{f}(F)\le g(q, r)$.
\end{lem}

\begin{proof}
We may assume that $V(F)=V(H\boxtimes P)$ because adding isolated vertices preserves the above three conditions and $\reduced{f}(F)$. Say $P=(w_1,w_2,\dots,w_{\ell})$. Let $\cT:=(T, \mathcal{B})$ and let $R$ be the root bag of $\cT$.
We proceed by induction on the number of bags of $\cT$. Since the root bag is empty, $\cT$ has at least two bags.

First suppose that $\cT$ consists of exactly two bags. 
Since $(T,\mathcal{B})$ is $(k,q)$-rooted, $|V(H)|\leq q$. 
By the neighbourhood condition, the underlying graph of $F$  is isomorphic to a subgraph of $S_{\ell,q,r}$. By assumption, $f(S_{\ell,q,r})\le g(q,r)$, and since $f$ is monotone, for every subgraph $Y$ of $S_{\ell,q,r}$, $f(Y)\le g(q, r)$. We obtain a reduction sequence of $F$ as follows. For $i=1,\dots,\ell-1$, arbitrarily identify $V(H)\times \{w_i\}$ into a vertex,  and then identify the resulting vertex with a vertex in $V(H)\times \{w_{i+1}\}$. Lastly, we identify $V(H)\times \{w_\ell\}$ into a vertex. The underlying graph of every trigraph in this reduction sequence is isomorphic to a subgraph of $S_{\ell, q,r}$, which shows that $\reduced{f}(F)\le g(q,r)$.

Now assume that $\cT$ has at least three bags. Let $B$ be an internal bag at maximum distance in $\cT$ from $R$. So all the children of $B$ are leaf bags. If $B=R$, then let $Y:=\emptyset$; otherwise, let $Y:=B\cap B'$ where $B'$ is the parent of $B$. 

First suppose that $B$ has at least two child bags $Q$ and $Q'$. 
If $\abs{(Q\cup Q')\setminus B}\le q$, then 
we obtain a tree-decomposition from $\cT$ by removing $Q$ and $Q'$ and attaching a leaf bag $Q\cup Q'$ to $B$. This results in a new tree-decomposition that satisfies the given conditions and has one fewer bag. So we are done by  induction. Now assume that 
$\abs{(Q\cup Q')\setminus B}> q$.

Since  $(T, \cB)$ is $(k,q)$-rooted, $\abs{(Q\cup Q')\setminus B}\le  2q$.
Furthermore, since $V(F)=V(H\boxtimes P)$, for each $w\in V(P)$, we have $\abs{((Q\cup Q')\setminus B)\times \{w\}}=\abs{(Q\cup Q')\setminus B}> q$.
Let 
$C:=Q\cup Q'\cup B$ and 
$D:=(V(H)\setminus (Q\cup Q')) \cup B$.
Note that $(C, D)$  is a rooted separation of $H$ at $B$. 
By the separation condition, for each $w\in V(P)$, 
 \[\left| \left\{N_{F}(v)\cap (D\times V(P)):v\in ((C\setminus D)\times \{w\})\cap V(F)\right\}  \right| \le q.\] 
So, for each $w\in V(P)$, there are distinct vertices $y$ and $z$ in $((Q\cup Q')\setminus B)\times \{w\}$ having the same neighbourhood on $D\times V(P)$. 

For $i=1,2,\dots,\ell$, reduce $((Q\cup Q')\setminus B)\times \{w_i\}$ into a set of $q$ vertices, by repeatedly choosing two vertices having the same neighbourhood on $D\times V(P)$. Note that we create no red edge incident with $D\times V(P)$.

Suppose that $U$ is the red graph constructed immediately  after identifying some vertices of $((Q\cup Q')\setminus B)\times \{w_i\}$ for some $i$. We claim that  $f(U)\le g(q,r)$.

\begin{figure}[!ht]
\centerline{\includegraphics[scale=0.3]{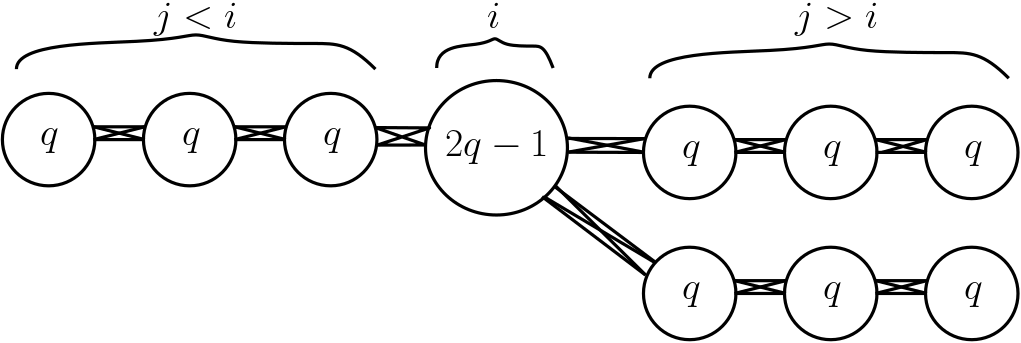} }
\caption{Identifications on $\bigcup_{j\in \{1, 2, \ldots, \ell\}} ((Q\cup Q')\setminus B)\times \{w_j\}$, when $r=1$.}
\label{fig:sxq2}
\end{figure}

For each $j<i$, $((Q\cup Q')\setminus B)\times \{w_j\}$ has been identified to a set of $q$ vertices, say $W_j$. For each $j>i$, $(Q\setminus B)\times \{w_j\}$ and $(Q'\setminus B)\times \{w_j\}$ are not yet identified, and so there are only black edges between $\bigcup_{j>i} (Q\setminus B)\times \{w_j\}$ and $\bigcup_{j>i} (Q'\setminus B)\times \{w_j\}$ by the red edge condition.  Also, $((Q\cup Q')\setminus B)\times \{w_i\}$ has been identified to at most $2q-1$ vertices, because at least one pair of vertices is identified. Call it $W_i$.  See \cref{fig:sxq2} for an illustration. 

Let 
\[A_1:=\bigcup_{j\le i}W_i\cup \left(\bigcup_{j>i} (Q\setminus B)\times \{w_j\}\right) \cup \left(\bigcup_{j>i} (Q'\setminus B)\times \{w_j\}\right), \]
and let $A_2:=V(U)\setminus A_1$. Observe that $A_2=D\times V(P)$.
Since we create no red edge incident with $D\times V(P)$ during the identifications, there is no edge between $A_1$ and $A_2$ in $U$. Furthermore, by the red edge condition and the neighbourhood condition, for each component of $U[A_2]$, its underlying graph is a subgraph of $S_{\ell,q,r}$.
Thus, $f(U[A_2])\le g(q,r)$. 
Also, the underlying graph of $U[A_1]$ is a subgraph of $S_{\ell, q, r}$ where $W_i$ is a subset of the center of $S_{\ell, q,r}$. So, $f(U[A_1])\le g(q,r)$. Since $f$ is union-closed, $f(U)\le g(q,r)$.

Now we explain how to apply the induction hypothesis to the resulting trigraph.
Let $F'$ be the resulting trigraph. Let $H'$ be the graph obtained from $H$ by removing $(Q\cup Q')\setminus B$ and adding a clique $Z$ of size $q$ that is complete to $B$. We obtain a tree-decomposition $(T', \cB')$ of $H'$ from $(T, \cB)$ by removing bags $Q$ and $Q'$, and adding a new bag $Z\cup B$ incident with $B$ (and $|(Z\cup B)\setminus B|\leq q$ as desired). Observe that $F'$ is a trigraph with $V(F')\subseteq V(H'\boxtimes P)$ satisfying the red edge and neighbourhood conditions. 

We verify that 
$F'$, $H'$, and $(T', \cB')$ satisfy the separation condition. Let $(C', D')$ be a rooted separation of $F'$ of $H'$ from $(T', \cB')$, and let $z\in V(P)$. 

\textbf{Case 1.} $Z\subseteq C'$: 

Let $(C_{pre}, D')$ be the rooted separation of $H$ obtained from $(C', D')$ by removing $Z$ from $C'$ and adding $(Q\cup Q')\setminus B$. Since we identified two vertices in $(C_{pre}\setminus D')\times \{z\}$ that have the same neighbourhood on $D' \times V(P)$, 
\begin{align*}
    &\left| \left\{N_{F'}(v)\cap (D'\times V(P)):v\in ((C'\setminus D')\times \{z\})\cap V(F')\right\}  \right| \\
    \le &\left| \left\{N_{F}(v)\cap (D'\times V(P)):v\in ((C_{pre}\setminus D')\times \{z\})\cap V(F)\right\}  \right| \le q.
\end{align*}

\textbf{Case 2.} $Z\subseteq D'$:

Let $(C', D_{pre})$ be the rooted separation of $H$ obtained from $(C', D')$ by removing $Z$ from $D'$ and adding $(Q\cup Q')\setminus B$. Again, since we identified two vertices in $(D_{pre}\setminus C')\times \{z\}$ that have the same neighbourhood on $C'\times V(P)$, 
\begin{align*}
    &\left| \left\{N_{F'}(v)\cap (D'\times V(P)):v\in ((C'\setminus D')\times \{z\})\cap V(F')\right\}  \right| \\
    \le &\left| \left\{N_{F}(v)\cap (D_{pre}\times V(P)):v\in ((C'\setminus D_{pre})\times \{z\})\cap V(F)\right\}  \right| \le q.
\end{align*}

Thus, $F'$, $H'$, and $(T', \cB')$ satisfy the separation condition.

Since $(T', \cB')$ has one fewer bag than $(T, \cB)$, by induction, there is a reduction sequence $L'$ of~$F'$, where for every trigraph $G$ in $L'$, $f(\widetilde{G})\le g(q,r)$. Together with the partial reduction sequence producing $F'$ from $F$, this gives the desired reduction sequence for $F$.

To finish the proof, it remains to consider the case in which $B$ has exactly one child  bag $Q$. Since $\cT$ has at least three bags, $B$ has its parent $B'$ and $Y=B\cap B'$. Now, $(B\cup Q)\setminus Y$ has at most $q+k+1\le 2q$ vertices, and we can do the same procedure in Cases 1 or 2 to reduce $((B\cup Q)\setminus Y)\times \{w\}$ (for each $w\in V(P)$) to a set of at most $q$ vertices by identifying two vertices having the same neighbourhood on $\bigcup_{j\in \{1, 2, \ldots, \ell\}} Y\times \{w_j\}$. This will correspond to vertices of $H$ forming one bag with $Y$.
Note that at the beginning, all edges between $\bigcup_{j\in \{1, 2, \ldots, \ell\}} ((B\cup Q)\setminus Y)\times \{w_j\}$ and $\bigcup_{j\in \{1, 2, \ldots, \ell\}} Y\times \{w_j\}$ are black.
So, the underlying graphs of red graphs constructed from $\bigcup_{j\in \{1, 2, \ldots, \ell\}} ((B\cup Q)\setminus Y)\times \{w_j\}$ will be subgraphs of $S_{\ell,q,r}$.

Let $F'$ be the resulting trigraph. Let $H'$ be the graph obtained from $H$ by removing $(B\cup Q)\setminus Y$ and adding a clique $Z$ of size $q$ that is complete to $Y$.
We obtain a tree-decomposition $(T', \cB')$ of $H'$ from $(T, \cB)$ by removing bags $Q$ and $B$, and adding a new bag $Z\cup Y$ incident with $B'$ (and $|Z\setminus B'|\leq q$ as desired). 
It is not difficult to see that $F', H', (\cT', \cB')$ satisfy the red edge, separation, and neighbourhood conditions. 
So, we can apply induction, which completes the proof of the lemma.
 \end{proof}

We now rewrite \cref{ProductPath} in a more useful form. 

\begin{lem}
\label{RowTreewidthNeighbourhoodPowers} 
Let $k,r,\pi^*\in\mathbb{N}$. Let $G$ be a graph with row-treewidth $k$, such that $\pi^r_G(X) \leq \pi^* $ for every set $X\subseteq V(G)$ with $|X|\leq (2r+1)(k+1)$. Then 
\[\reduced{\bw}(G^r) \leq (2r+2)\pi^*-2
\quad\text{and}\quad
\tww(G^r) \leq (3r+2)\pi^*-2.\]
\end{lem}

\begin{proof}
We may assume that $G\subseteq H\boxtimes P$, where $\tw(H)\leq k$ and $P$ is a path. We now show that \cref{ProductPath} is applicable, where $F$ is the trigraph obtained from $G^r$ with no red edges. The red edge condition holds trivially. 

Consider a  rooted tree-decomposition $(T, \cB)$ of $H$ with width at most $k$. Let $(C, D)$ be a rooted separation of $(T, \cB)$. Consider $z\in V(P)$ and $v\in ((C\setminus D)\times \{z\})\cap V(F)$. Every path in $G$ from $v$ to $D\times V(P)$ with length at most $r$ must intersect $(C\cap D)\times N^r_P[z]$. Thus $N_{F}(v)\cap (D\times V(P))$ is determined by the distance-$r$ profile of $v$ on $(C\cap D) \times N^r_{P}[z]$, which has at most $(2r+1)(k+1)$ vertices. Thus
\begin{align*}
  \left| \left\{N_{F}(v)\cap (D\times V(P)):v\in ((C\setminus D)\times \{z\})\cap V(F)\right\}  \right| 
\,\le\, 
\pi^r_G((C\cap D) \times N^r_{P}[z]) 
\,\leq\, \pi^*.
\end{align*}
Hence the separation condition in \cref{ProductPath} holds with $q=\pi^*$. The neighbourhood condition holds since $N_F[v]=N^r_G[v]$ and $G\subseteq H\boxtimes P$.

For the upper bound on $\reduced{\bw}$, we may apply \cref{ProductPath} with $f=\bw$ and $g(q,r)=(2r+2)q-2$ by \cref{bwS}. Thus $\reduced{\bw}(G)\leq g(q,r) = (2r+2)\pi^*-2$. 

For the upper bound on $\tww$, we may apply \cref{ProductPath} with $f=\Delta$ and $g(q,r)=(3r+2)q-2$ by \cref{DeltaS}. Thus $\tww(G)\leq g(q,r) = (3r+2)\pi^*-2$.
\end{proof}

Since $\pi^r_G(A)\leq(r+1)^{|A|}$ and by \eqref{equ:profile}, \cref{RowTreewidthNeighbourhoodPowers} is applicable with 
$\pi^*=(r+1)^{(2r+1)(k+1)}$. Thus:

\begin{cor}
\label{RowTreewidPowers} 
For every graph $G$ with row-treewidth $k$ and for $r\in\mathbb{N}$,
$$\reduced{\bw}(G^r) \leq 2 (r+1)^{(2r+1)(k+1)+1} -2
\quad\text{and}\quad
\tww(G^r) \leq (3r+2)(r+1)^{(2r+1)(k+1)}-2.$$
\end{cor}

\cref{RowTreewidPowers,GenusProductStructure} imply:

\begin{cor}
\label{GenusPowers} 
For every graph $G$ with Euler genus $\gamma$ and for $r\in\mathbb{N}$,
$$\reduced{\bw}(G^r) \leq 2 (r+1)^{(2r+1)(2\gamma+7)+1} -2
\quad\text{and}\quad
\tww(G^r) \leq (3r+2)(r+1)^{(2r+1)(2\gamma+7)}-2.$$
In particular, for every planar graph $G$ and $r\in\mathbb{N}$,
$$\reduced{\bw}(G^r) \leq 2 (r+1)^{14r+8} -2
\quad\text{and}\quad
\tww(G^r) \leq (3r+2)(r+1)^{14r+7}-2.$$
\end{cor}

Applying the neighbourhood complexity bounds from \cref{FirstNeigh} we obtain the following improved bounds in the $r=1$ case.

\begin{thm}
\label{ReducedBandwidthGenus}
\label{TwinwidthGenus}
\label{PlanarReducedBandwidth}
\label{PlanarTwinwidth}
For every graph $G$ with Euler genus $\gamma$,
$$\reduced{\bw}(G)\leq 164\gamma+466\quad\text{and}\quad
\tww(G)\leq 205\gamma+583.$$
In particular, for every planar graph $G$,
\[\reduced{\bw}(G) \leq 466\quad\text{and}\quad\tww(G) \leq 583.\]
\end{thm}

\begin{proof}
By \cref{GenusProductStructure}, $G$ has row-treewidth at most $k:=2\gamma+6$. By \cref{SurfaceNeighbours}, for every $X\subseteq V(G)$ we have $\pi^1_G(X)\le \max\{4,6|X|+5\gamma-9\}$, which is at most $41\gamma+117$ when $|X| \le 3(k+1)=6\gamma+21$. The result thus follows from \cref{RowTreewidthNeighbourhoodPowers} with $r=1$ and $\pi^*=41\gamma+117$.
\end{proof}

\begin{thm}
\label{ReducedBandwidthTwinwidthCol}
For every graph $G$ with row-treewidth $k$ and $\col_5(G)\leq c$,
\[
\reduced{\bw}(G) \leq (12k+16) \, 2^{c-1} 
\quad\text{and}\quad
\tww(G) \leq (15k+20)\, 2^{c-1}\]
\end{thm}

\begin{proof}
We may assume that $c\geq 1$. 
Let $X\subseteq V(G)$ with $|X|\leq 3(k+1)$.
If $|X|\geq c-1$ then by \cref{NeighCol}, 
$ \pi^1_G(X) \leq  2^{c-1}( |X|-c+2) \leq (3k+4)\,2^{c-1}$,  
otherwise
$ \pi^1_G(X) \leq  2^{|X|}\leq (3k+4)\,2^{c-1}$
by \eqref{equ:profile}. The result thus follows from \cref{RowTreewidthNeighbourhoodPowers} with $r=1$ and 
$\pi^*= (3k+4)\,2^{c-1}$
\end{proof}

\begin{thm}
\label{gkPlanar}
Every $(\gamma,k)$-planar graph $G$ has reduced bandwidth, 
$$\reduced{\bw}(G)\leq 2^{O(\gamma k)}.$$
\end{thm}

\begin{proof}
By \cref{gkPlanarProduct}, $G$ has row-treewidth $O(\gamma k^6)$. Let $X\subseteq V(G)$ with $|X|\leq 3(k+1)$. If $|X| \geq 22(2\gamma+3)(k+1)-1$, then by \cref{NeighGammak}, we have
$\pi^1_G(X) \leq 2^{22(2\gamma+3)(k+1)-1}(3k+3)$, otherwise  $\pi^1_G(X)\le 2^{|X|}\le  2^{22(2\gamma+3)(k+1)-1}(3k+3)$
by \eqref{equ:profile}. The result thus follows from \cref{RowTreewidthNeighbourhoodPowers} with $r=1$ and  $\pi^*=2^{O(\gamma k)}$.
\end{proof}

We have the following result for squares of graphs of given Euler genus. 

\begin{thm}
\label{GenusSquare}
For every graph $G$ of Euler genus $\gamma$,
\begin{align*}
    \reduced{\bw}(G^2) & \leq
234000 \gamma^4 + 1983300 \gamma^3  + 5769330 \gamma^2   + 6671550 \gamma + 2711084
\quad\text{and}\\
\tww(G^2) & \leq 
312000 \gamma ^4
+ 2644400 \gamma ^3
 + 7692440 \gamma ^2
 + 8895400 \gamma 
 + 3614782.
\end{align*}
In particular, for every planar graph $G$,
\begin{align*}
    \reduced{\bw}(G^2)  \leq 2711084
\quad\text{and}\quad
\tww(G^2) & \leq  3614782.
\end{align*}
\end{thm}

\begin{proof}
By \cref{GenusProductStructure}, $G$ has row-treewidth at most $k:=2\gamma+6$. By \cref{SurfaceNu2}, for every set $X\subseteq V(G)$,
\[ \pi^2_G(X) \leq \max\{ 2, (6 |X| + 5\gamma -9) ( (60 \gamma^2 + 125 \gamma + 68)|X|-120 \gamma^2 -250 \gamma -132) \}. \]
The result thus follows from \cref{RowTreewidthNeighbourhoodPowers} with $r=2$ and 
\begin{align*}
\pi^* 
& = 
( 6 ( 5(k+1) ) + 5\gamma -9) ( (60 \gamma^2 + 125 \gamma + 68) 5(k+1) -120 \gamma^2 -250 \gamma -132)\\
& = 
39000 \gamma^4 + 330550 \gamma^3+ 961555 \gamma^2 + 1111925 \gamma + 451848 .\qedhere
\end{align*}
\end{proof}

We now give an example of the application of \cref{GenusSquare}. 
Map graphs are defined as follows. Start with a graph $G_0$ embedded in a surface of Euler genus $\gamma$, with each face labelled a `nation' (or a `lake'). Let $G$ be the graph whose vertices are the nations of $G_0$, where two vertices are adjacent in $G$ if the corresponding faces in $G_0$ share a vertex. Then $G$ is called a \defn{map graph} of Euler genus $\gamma$. A map graph of Euler genus 0 is called a (plane) \defn{map graph}; see \citep{FLS-SODA12,CGP02} for example. Map graphs generalise graphs embedded in surfaces, since a graph $G$ has Euler genus at most $\gamma$ if and only if  $G$ has a representation as a map graph of Euler genus at most $\gamma$ with the extra property that each vertex of $G_0$ is incident with at most three nations \citep{DEW17}. Our results for map graphs are independent of the number of nations incident to each vertex of $G_0$.  It follows from the definition that for each map graph $G$ of Euler genus $\gamma$, there is a bipartite graph $H$ with bipartition $\{A,B\}$, such that $H$ has Euler genus at most $\gamma$ and $V(G) = A$ and $vw \in E(G)$ whenever $v,w \in N_H(b)$ for some $b \in B$ (see \citep{DEW17}); $G$ is called the \defn{half-square} of $H$. Note that $G$ is an induced subgraph of $H^2$. Since reduced bandwidth is hereditary\footnote{It is an easy exercise to show that $\reduced{f}$ is hereditary for every monotone graph parameter $f$; see \citep[Section~4.1]{TW-I} for a proof sketch in the case of twin-width.}, the results in \cref{GenusSquare} also hold for map graphs of Euler genus $\gamma$.  We emphasise there is no dependence on the maximum degree (which is customary when considering map graphs). 

The proof of \cref{ProductPath} is constructive and the desired reduction sequence can be found in time $O(|V(F)|^2)$ when $H$ and $P$ are given. For all the above examples, the graphs $H$ and $P$ can be found in $O(|V(G)|^2)$ time; see \citep{DJMMUW20,Morin21,BMO22} for details and speed-ups. So the desired reduction sequence can be found in  $O(|V(G)|^2)$ time.

\section{Proper Minor-Closed Classes}
\label{ProperMinorClosedClasses}

This section shows that fixed powers of graphs in any proper minor-closed class have bounded reduced bandwidth. To use the product structure theorem for $K_t$-minor-free graphs (\cref{MinorFree}), we first prove an upper bound on the reduced bandwidth of the $r$-th power of a subgraph of $(H\boxtimes P)+K_a$  where $H$ has bounded treewidth, $P$ is a path, and $a\in \mathbb{N}$. To do so, in \cref{lem:productpathpowerrpower} below we prove an extension of \cref{ProductPath} that allows for apex vertices. 
Consider the $r$-th power of a subgraph of $(H\boxtimes P)+K_a$. There may exist an edge $vw$ with $v,w\in V(H)\times V(P)$ where the corresponding vertices in $P$ are far apart in $P$, because there may exist a path of length at most $r$ through some vertices in $K_a$. So the neighbourhood condition of \cref{ProductPath} is no longer relevant. Instead, red edges are controlled by a new `red edge condition', and black edges are controlled by two separation conditions. The second separation condition is necessary to deal with the base case when the tree-decomposition of $H$ has one non-empty bag.

\begin{lem}\label{lem:productpathpowerrpower}
Let $a, k, q, r\in\NN$ with $q\ge k+1$. 
Let $f$ be a monotone and union-closed graph parameter and let $g:\mathbb{N}\times \mathbb{N}\to\mathbb{R}$ be a function such that $f(S_{x,q,r})\le g(q,r)$ for all $x\in \mathbb{N}$.
Let $P=(w_1,w_2,\dots,w_{\ell})$ be a path and let $H$ be a graph admitting a $(k,q)$-rooted tree-decomposition $(T, \cB)$.
 Let $F$ be a trigraph with  $V(F)\subseteq V((H\boxtimes P)+K_a)$ such that: 
   \begin{itemize}
   \item \textup{(\defn{red edge condition})} for every red edge $vw$ of $F$, there is a leaf bag $B$ with parent $B'$ in $(T, \cB)$ and there are vertices $x,y$ in $P$ with $\dist_P(x,y)\le r$, such that $v,w\in (B\setminus B')\times \{x,y\}$, 
  \item \textup{(\defn{first separation condition})} for every rooted separation $(C, D)$ of $H$ from $(T, \cB)$ and every $z\in V(P)$, 
 \[\left| \left\{N_{F}(v)\cap M_1:v\in ((C\setminus D)\times\{z\})\cap V(F)\right\}  \right| \le q,\] where $M_1=\Big(D\times V(P)\Big) \cup \Big(C\times (V(P)\setminus N^r_P[z])\Big)\cup V(K_a)$,
 \item \textup{(\defn{second separation condition})} for every $t\in \{1,\ldots, \ell\}$, 
 \[|\{N_F(v)\cap M_2: v\in (V(H)\times \{w_1,\dots,w_t\})\cap V(F)\}|\le q,\] where $M_2=\big(V(H)\times \{w_{t+r+1},\dots,w_\ell\}\big)\cup V(K_a)$.
 \end{itemize}
 Then:
 \begin{enumerate}[(1)]
     \item there is a partial reduction sequence identifying $(V(H)\times V(P))\cap V(F)$ into at most $q$ vertices such that for every trigraph $G$ in the sequence, $G$ has no red edge incident with $V(K_a)$ and $f(\widetilde{G})\le g(q,r)$, and
     \item  $\reduced{f}(F)\le  g(a+q,r)$.
 \end{enumerate}
  \end{lem}
\begin{proof}
     We identify the $V(H)\times V(P)$ part using a partial reduction sequence similar to the one in the proof of~\cref{ProductPath}. The main difference in the induction step is that when we consider a rooted separation $(C,D)$ from $(T, \cB)$ and $z\in V(P)$, instead of choosing two vertices having the same neighbourhood on $D\times V(P)$, here we choose two vertices having the same neighbourhood on \[\Big(D\times V(P)\Big) \cup \Big(C\times (V(P)\setminus N^r_P[z])\Big)\cup V(K_a).\] This clearly preserves the new red edge and separation conditions. So, as in the proof of~\cref{ProductPath}, choose an internal bag $B$ at maximum distance in $(T, \cB)$ from the root bag. Then for $i=1,2,\dots,\ell$, merge $(Q\setminus B)\times \{w_i\}$ and $(Q'\setminus B)\times \{w_i\}$ for each $w_i\in V(P)$ if $B$ has two child bags $Q$ and $Q'$, and otherwise, merge $(B\setminus B')\times \{w_i\}$ and $(Q\setminus B')\times \{w_i\}$ for the unique child $Q$ of $B$ and the parent $B'$ of $B$.
     Because of the choice of identified vertices, for each red component of an intermediate trigraph, its underlying graph is isomorphic to a subgraph of $S_{\ell, q, r}$, and has $f$ at most $g(q,r)$.
    
    Now, we consider the base case where $(T, \cB)$ has two bags $R$ and $B$, where $R$ is the empty root bag.
     We need a new argument for this case, because if we arbitrarily identify two vertices, then we may create some red edges between two vertices contained in $V(H)\times\{x\}$ and $V(H)\times \{y\}$ where $x$ and $y$ are far from each other in $P$, and we cannot guarantee that the red graphs have bounded $f$-values.
     
     We prove by induction on $\ell$ that there is a partial reduction sequence identifying $(V(H)\times V(P))\cap V(F)$ into at most $q$ vertices, such that for every trigraph $G$ in the sequence, \begin{itemize}
        \item $G$ has no red edge incident with $V(K_a)$, and
    $f(\widetilde{G})\le g(q,r)$. 
     \end{itemize}
    We may assume that $H$ is a complete graph on $q$ vertices. If $\ell=1$, then $V(H)\times V(P)$ has at most $q$ vertices, and there is nothing to show. 
    
    We assume that $\ell>1$.
    If $|(B\times \{w_1\})\cup (B\times \{w_2\})|\le q$, then remove $B\times \{w_1\}$ and replace $B\times \{w_2\}$ with $(B\times \{w_1\})\cup (B\times \{w_2\})$. Thus we can reduce the length of $P$, and are done by  induction.
    Now assume that $|(B\times \{w_1\})\cup (B\times \{w_2\})|> q$.
    By the second separation condition, 
    there are two vertices in $(B\times \{w_1\})\cup (B\times \{w_2\})$
    that have the same neighbourhood on \[M_2=\big(V(H)\times \{w_{2+r+1},\dots,w_\ell\}\big)\cup V(K_a)\]
    By repeatedly identifying two vertices having the same neighbourhoods on $M_2$, we identify $(B\times \{w_1\})\cup (B\times \{w_2\})$ into at most $q$ vertices. In this way, we can reduce the length of $P$, and we are done by induction. This proves (1).
    
    For (2), observe that the statement (1) holds when we replace $g(q,r)$ with $g(a+q, r)$, because every subgraph of $S_{\ell, q, r}$ is also a subgraph of $S_{\ell, a+q, r}$. We obtain the desired reduction sequence by arbitrarily identifying the resulting trigraph on at most $a+q$ vertices into a 1-vertex graph.
 \end{proof}

 We extend \cref{lem:productpathpowerrpower} to deal with the internal node case of the tree-decomposition of $X$-minor free graphs. We now allow bags of size more than $q$, and we start by identifying vertices corresponding to those bags. The first three conditions in \cref{lem:productpathpowerrpower2} are the same as the conditions in \cref{lem:productpathpowerrpower}. We use \cref{lem:productpathpowerrpower} as a base case.

\begin{lem}\label{lem:productpathpowerrpower2}
Let $a, k, q, r, t\in\NN$ with $q\ge k+1$. Let $f$ be a monotone and union-closed graph parameter. Let $g:\mathbb{N}\times \mathbb{N}\to\mathbb{R}$ be a function such that $f(S_{x,q,r})\le g(q,r)$  for all $x\in \mathbb{N}$. Let $P=(w_1,w_2,\dots,w_{\ell})$ be a path and let $H$ be a graph admitting a $(k,\infty)$-rooted tree-decomposition $(T, \cB)$. Let $F$ be a trigraph with  $V(F)\subseteq V((H\boxtimes P)+K_a)$ such that the red edge, first separation and second separation conditions from \cref{lem:productpathpowerrpower} are satisfied, and in addition: 
 \begin{itemize}
 \item \textup{(\defn{large leaf bag condition})} 
 for every leaf bag $B$ with parent bag $B'$ and $|B\setminus B'|>q$, 
 \begin{itemize}
     \item $(B\times V(P))\cap V(F)\subseteq B\times \{x,y\}$ for some adjacent vertices $x,y$ of $P$,
     \item there is a partial reduction sequence identifying $((B\setminus B')\times V(P))\cap V(F)$ into at most $q$ vertices such that for every trigraph $G$ in the reduction sequence, $G$ has no red edge incident with $((V(H)\setminus (B\setminus B'))\times V(P))\cup V(K_a)$ and $f(\widetilde{G})\le t$.
 \end{itemize} 
 \end{itemize}
  Then:
 \begin{enumerate}[(1)]
     \item There is a reduction sequence identifying $(V(H)\times V(P))\cap V(F)$ into at most $q$ vertices such that for every trigraph $G$ in the reduction sequence, $G$ has no red edge incident with $V(K_a)$ and $f(\widetilde{G})\le \max (g(q,r), t)$.
     \item  $\reduced{f}(F)\le \max (g(a+q,r), t)$.
 \end{enumerate}
 \end{lem}
 
 \begin{proof}
 We say that a leaf bag $B$ of $(T, \cB)$ with parent bag $B'$ is \defn{large} if $|B\setminus B'|> q$.
 We prove (1) by induction on the number of large leaf bags.
 If there are no large leaf bags, then $(T, \cB)$ is $(k,q)$-rooted, and the result follows from \cref{lem:productpathpowerrpower}.
 Now assume that $(T, \cB)$ contains a large leaf bag.
 
 Let $B$ be a large leaf bag with parent bag $B'$.
 Let $U_1:=(B\setminus B')\times V(P)$ and $U_2:=((V(H)\setminus (B\setminus B'))\times V(P))\cup V(K_a)$.
 
 Let $L_B$ be a partial reduction sequence identifying $U_1\cap V(F)$, as described in the large leaf bag condition. By the red edge condition, there is no red edge between $U_1$ and $U_2$. Apply $L_B$ to $U_1\cap V(F)$. By assumption, we always identify two vertices having the same neighbourhoods on $U_2$, and therefore, it creates no red edge incident with $U_2$. Since $(B\times V(P))\cap V(F)\subseteq B\times \{x,y\}$ for some adjacent vertices $x,y$ of $P$, the resulting trigraph satisfies the red edge condition. Also, each red graph of an intermediate trigraph constructed from $U_1$ has $f$ at most $t$, by assumption. 
  
Let $F'$ be the resulting trigraph. Let $H'$ be the graph obtained from $H$ by removing $B\setminus B'$ and adding a clique $Z$ of size $q$ that is complete to $B'$. We obtain a tree-decomposition $(T', \cB')$ of $H'$ from $(T, \cB)$ by removing the bag $B$, and adding a new bag $Z\cup B'$ incident with $B'$ (and $|(Z\cup B')\setminus B'|\leq q$ as desired). Observe that $F'$ is a trigraph with $V(F')\subseteq V(H'\boxtimes P)\cup V(K_a)$ satisfying the four conditions. 

Since $(T', \cB')$ has one fewer large leaf bag than $(T, \cB)$, by induction, there is a reduction sequence identifying $(V(H')\times V(P))\cap V(F')$ into at most $q$ vertices such that for every trigraph $G$ in the reduction sequence, $G$ has no red edge incident with $V(K_a)$ and $f(\widetilde{G})\le \max (g(q,r), t)$. Together with the reduction sequence producing $F'$ from $F$, we obtain a desired reduction sequence for $F$.

Similar to \cref{lem:productpathpowerrpower},
    the statement (1) holds when we replace $g(q,r)$ with $g(a+q, r)$, because every subgraph of $S_{\ell, q, r}$ is also a subgraph of $S_{\ell, a+q, r}$. We obtain the desired reduction sequence by arbitrarily identifying the resulting trigraph on at most $a+q$ vertices into a 1-vertex graph.
    \end{proof}

\begin{thm}\label{thm:Hminorfreepower}
 Let $t, r\in\NN$. Let $f$ be a monotone and union-closed graph parameter  and let $g:\mathbb{N}\times \mathbb{N}\to\mathbb{R}$ be a function such that $f(S_{x,q,r})\le g(q,r)$ for all $x\in \mathbb{N}$.
Then there exists $c\in\NN$ such that $\reduced{f}(G^r)\leq c$ for every $K_t$-minor-free graph $G$. 
\end{thm}

\begin{proof}
By \cref{MinorFree}, there exist $k,a\in\mathbb{N}$ such that $G$ has  a tree-decomposition $\cT=(T, \cB)$ in which every torso of a bag is a subgraph of $(H\boxtimes P)+K_a$ for some graph $H$ of treewidth at most $k$  and some path $P$. Consider $\mathcal{T}$ to be rooted at an arbitrary bag $R$ of $\cT$. Let $\beta:= 2(k+1)+a$. Note that the size of a maximum clique of $(H\boxtimes P)+K_a$ is at most $\beta$.
Let $h$ be the length of the longest path from an internal bag to $R$ in $\cT$. 
Define
\[\theta:=a+\tbinom{\beta}{2}(r+1)+\beta r,
\quad
q:=(r+1)^{(k+1)(2r+1)+\theta}
\quad \text{and} \quad
c:=g(q,r).\]

For each bag $B$ of $\cT$,  if $B=R$ then let $Y_B:=\emptyset$; otherwise, let $Y_B:=B\cap B'$ where $B'$ is the parent of $B$.
Also, let $U_B$ be the union of $B$ and all the descendants of $B$. Note that $0\le h-\dist_T(B,R)\le h$.

We prove by induction on $h-\dist_T(B, R)$ that 
\begin{itemize}
    \item[($\ast$)] there is a partial reduction sequence identifying $U_B\setminus Y_B$ into at most $q$ vertices in $G^r$ such that:
    \begin{itemize}
        \item we always identify two vertices $u$ and $v$ where the corresponding vertices in $G$ have the same distance-$r$ profiles on $Y_B$ in $G[U_B]$, and
        \item for every trigraph $G^*$ in the sequence,
    $f(\widetilde{G^*})\le c$.
    \end{itemize}
\end{itemize}

Let $B$ be a bag. Let $G_1$ be the graph obtained from $G[U_B]$ by adding, for each pair of vertices $y_1, y_2$ of $Y_B$ with $\dist_{G-E(G[U_B])}(y_1, y_2)\le r$, a new path $X_{y_1, y_2}$ of length $\dist_{G-E(G[U_B])}(y_1, y_2)$ whose end-vertices are $y_1$ and $y_2$, and then adding, for every vertex $y$ of $Y_B$, a path $W_y$ of length $r-1$ whose end-vertex is $y$.

We claim that for any two vertices $v,w\in U_B$,
\begin{itemize}
    \item $\dist_G(v,w)\le r$ if and only if $\dist_{G_1}(v,w)\le r$.
\end{itemize} 

Assume that there is a path $Z$ of length at most $r$ between $v$ and $w$ in $G$. We construct a path $Z^*$ of $G_1$ from $Z$ as follows.
\begin{itemize}
    \item We repeatedly choose a maximal subpath $Q$ of $Z$ whose end-vertices are in $U_B$ and all internal vertices are not in $U_B$. Assume its end-vertices are $y_1$ and $y_2$. Then we replace $Q$ with $X_{y_1,y_2}$.
\end{itemize}
By construction, $|E(X_{y_1, y_2})|=\dist_{G-E(G[U_B])}(y_1, y_2)\le |E(Q)|$. So, the resulting path $Z^*$ has length at most $r$, which shows that $\dist_{G_1}(v, w)\le r$. 

For the other direction, assume there is a path $Z$ of length at most $r$ between $v$ and $w$ in $G_1$. Similarly, we construct a walk of length at most $r$ in $G$ from $Z$ as follows.
\begin{itemize}
    \item We repeatedly choose a maximal subpath $Q$ of $Z$ of length at least $2$ whose end-vertices are in $U_B$ and all internal vertices are not in $U_B$.
    Note that $Q$ is $X_{y_1, y_2}$ for some $y_1, y_2\in Y$. By the construction, there is a path $Q_{y_1, y_2}$ in $G-E(G[U])$  of length $|E(X_{y_1,y_2})|$. Then we replace $Q$ with $Q_{y_1, y_2}$.
\end{itemize}
The resulting walk has length at most $r$ between $v$ and $w$ in $G$. Thus, $\dist_G(v,w)\le r$.

By the claim, $G^r[U_B]=(G_1)^r[U_B]$.

Now focus on $G_1$. First observe that $G_1[B]$ is a subgraph of the torso of $G[B]$ in $G$, and $G_1[B]$ is a subgraph of $(H\boxtimes P)+K_a$ for some graph $H$ of treewidth at most $k$ and some path $P$. Let $(T_B, \cB_B)$ be a rooted tree-decomposition of $H$ of width at most $k$. Let $P=(w_1,w_2,\dots,w_{\ell})$.

Observe that $Y_B\cup (\bigcup_{\{y_1, y_2\}\in \binom{Y_B}{2}}V(X_{y_1, y_2}))\cup (\bigcup_{y\in Y_B}V(W_y))$ contains at most  $\beta+\binom{\beta}{2}(r-1)+\beta(r-1)$ vertices. Since $G_1[B]$ is a subgraph of $(H\boxtimes P)+K_a$, the graph \[G_1[B]\cup \bigcup_{\{y_1, y_2\}\in \binom{Y_B}{2}} X_{y_1,y_2}\cup \bigcup_{y\in Y_B}W_y.\] 
can be regarded as a subgraph of $(H\boxtimes P)+K_{\theta}$, where 
$Y_B\cup (\bigcup_{\{y_1, y_2\}\in \binom{Y_B}{2}}V(X_{y_1, y_2}))\cup (\bigcup_{y\in Y_B}V(W_y))$ 
is placed in the $K_{\theta}$ part.  Let 
\[G_2:=G_1[B]\cup \bigcup_{\{y_1, y_2\}\in \binom{Y_B}{2} } X_{y_1,y_2}\cup \bigcup_{y\in Y_B}W_y.\]

Next, we extend the product structure for $G_2$ into one for $G_1$. We modify $H$ to $H^*$ as follows. Let $B'$ be a child of the bag $B$. Observe that $B\cap B'$ is a clique in the torso of $B$. Therefore, the vertices of $B\cap B'$ lie on $(N_B\times \{y_1,y_2\})\cup V(K_{\theta})$ for some consecutive two vertices $y_1$ and $y_2$ in $P$ and some bag $N_B$ of $(T_B, \cB_B)$.  
To $H$, we add a clique $Q_{B'}$ of size $|U_{B'}|$ that is complete to $N_B$. 
We add a leaf bag $Q_{B'}
\cup N_B$ adjacent to $N_B$, and embed the vertices of $U_{B'}\setminus B$ into $Q_{B'}\times \{y_1\}$. We do this process for every child of $B$. 

Let $H^*$ be the resulting graph obtained from $H$, and let $(T_B^*, \cB_B^*)$ be the resulting rooted tree-decomposition of $H^*$.
Since $(T_B, \cB_B)$ has width at most $k$, $(T_B^*, \cB_B^*)$ is $(k, \infty)$-rooted. Also 
$G_1$ is a subgraph of $(H^*\boxtimes P)+K_{\theta}$. 

We now apply \cref{lem:productpathpowerrpower2} to $(G_1)^r$ and $(T_B^*, \cB_B^*)$. To do so, we verify the four conditions. Since $(G_1)^r$ has no red edges, the red edge condition holds. 

(First separation condition) Let $(C,D)$ be a rooted separation of $H^*$ from $(T_B^*, \cB_B^*)$ and let $z\in V(P)$. Let \[M_1:=\Big(D\times V(P)\Big) \cup \Big(C\times (V(P)\setminus N^r_P[z])\Big)\cup V(K_\theta).\]
Observe that if there is a path of length at most $r$ in $G_1$ between $((C\setminus D)\times \{z\})\cap V(G_1)$ and $M_1$, 
this path intersects $((C\cap D)\times N^r_P[z])\cup V(K_\theta)$. Thus two vertices $v, w\in ((C\setminus D)\times \{z\})\cap V(G_1)$ having the same distance-$r$ profiles on 
$((C\cap D)\times N^r_P[z])\cup V(K_\theta)$ in $G_1$ have the same neighbourhood on $M_1$ in $(G_1)^r$. Therefore, 
\begin{align*}
    \left| \left\{N_{(G_1)^r}(v)\cap M:v\in ((C\setminus D)\times \{z\})\cap V(G_1)\right\}  \right|  
    \le\; & \pi^r_{G_1}\Big(((C\cap D)\times N^r_P[z])\cup V(K_\theta)\Big) \\
    \le\; & (r+1)^{(k+1)(2r+1)+\theta} \,=\, q.
\end{align*}

(Second separation condition) Let $t\in \{1, 2, \ldots, \ell\}$ and let \[M_2:=\big(V(H^*)\times \{w_{t+r+1},\dots,w_\ell\}\big)\cup V(K_\theta).\]
Observe that if there is a path of length at most $r$ in $G_1$ between $(V(H^*)\times \{w_1,\dots,w_t\})\cap V(G_1)$ and $M_2$, then this path intersects $V(K_{\theta})$. So, any two vertices in $(V(H^*)\times \{w_1,\dots,w_t\})\cap V(G_1)$ having the same distance-$r$ profiles on $V(K_\theta)$  have the same neighbourhood on $M_2$ in $(G_1)^r$.
Thus,
\begin{align*}
    |\{N_{(G_1)^r}(v)\cap M_2: v\in (V(H)\times \{w_1,\dots,w_t\})\cap V(G_1)\}| 
    \le\; \pi^r_{G_1}(V(K_\theta)) \,\le\, (r+1)^{\theta} \,\le\, q.
\end{align*} 

(Large leaf bag condition)
   Let $A$ be a leaf bag of $(T_B^*, \cB_B^*)$ with parent bag $A'$ with $|A\setminus A'|>q$.
   By the construction of $H^*$ and $(T_B^*, \cB_B^*)$, we know that $A\times V(P)\subseteq A\times \{y_1, y_2\}$ for some adjacent vertices $y_1$ and $y_2$ in $P$. 
   Also, $((A\setminus A')\times \{y_1, y_2\})\cap V(G_1)=U_{B'}\setminus B$ for some child $B'$ of $B$. 
    By induction, 
    there is a partial reduction sequence identifying $U_{B'}\setminus Y_{B'}$ into at most $q$ vertices in $G^r$ such that: 
    \begin{itemize}
        \item we always identify two vertices $u$ and $v$ where the corresponding vertices in $G$ have the same distance-$r$ profiles on $Y_{B'}$ in $G[U_{B'}]$, and
        \item for every trigraph $G^*$ in the sequence,
    $f(\widetilde{G^*})\le c$.
    \end{itemize}
    Since $G^r[U_B]=(G_1)^r[U_B]$, 
    this sequence is also a partial reduction sequence for $(G_1)^r[U_{B'}\setminus Y_{B'}]$.
    Note that if two vertices have the same distance-$r$ profiles on $Y_{B'}$ in $G[U_{B'}]$, then they have the same distance-$r$ profiles on $Y_B$ in $G[U_B]$.
    Thus, by the first bullet, 
    this sequence does not create any red edge incident with  $V(K_\theta)$.

   We deduce that $(G_1)^r$ and $(T_B^*, \cB_B^*)$ satisfy the large leaf bag condition of \cref{lem:productpathpowerrpower2}.

    Therefore, by \cref{lem:productpathpowerrpower2}, 
    there is a partial reduction sequence identifying $(V(H^*)\times V(P))\cap V(G_1)$ into at most $q$ vertices in $(G_1)^r$ such that for every trigraph $G^*$ in the sequence, $G^*$ has no red edge incident with $V(K_\theta)$, and $f(\widetilde{G^*})\le \max (g(q,r), c)$. 
    
    Now apply the same reduction sequence to $G^r$. We claim that we always identify two vertices having the same distance-$r$ profiles on $Y_B$ in $G[U_B]$. Suppose for contradiction that we identified two vertices $u$ and $v$ having distinct distance-$r$ profiles on $Y_B$ in $G[U_B]$.
    We may assume that for some $x\in Y_B$ and $\ell_1 \in \{1, 2, \ldots, r\}$ and $\ell_2\in \{\ell_1+1, \ldots r\}\cup \{\infty\}$, we have $\dist_{G[U_B]}(u,x)=\ell_1$ and $\dist_{G[U_B]}(v,x)=\ell_2$. But then on the path $W_x$, there is a vertex adjacent to $u$ but not to $y$ in $(G_1)^r$. So, when we identify these vertices in $(G_1)^r$, we create a red edge incident with $V(K_{\theta})$, a contradiction. Therefore, the claim holds. 
    
    We conclude that the statement $(\ast)$ holds by induction.  
    When $B=R$, we know that $U_B=V(G)$ and $Y_B=\emptyset$. Therefore, 
    there is a partial reduction sequence identifying $V(G)$ into at most $q$ vertices in $G^r$ such that
    for every trigraph $G^*$ in the sequence,
    $f(\widetilde{G^*})\le c$. We obtain the desired reduction sequence by arbitrarily identifying the resulting trigraph on at most $q$ vertices into a 1-vertex graph. For every trigraph $G^{**}$ in this sequence, $f(\widetilde{G^{**}})\le g(q,r)=c$, as every graph on at most $q$ vertices is a subgraph of $S_{1,q,r}$.
    This completes the proof of the theorem.
\end{proof}

\cref{thm:Hminorfreepower} with $f=\bw$ and $g(q,r)=(2r+2)q-2$ implies the following result, which is the main contribution of the paper. 
 
\begin{thm}
\label{MinorClosedPower}
For all $t,r\in\mathbb{N}$ there exists $c\in\NN$ such that for every $K_t$-minor-free graph $G$, \[ \reduced{\bw}(G^r)\leq c.\]
\end{thm}

\section{Expanders}
\label{Expanders}

This section shows that bounded degree expanders have unbounded reduced bandwidth. In fact, we show a stronger result for expanders excluding a fixed complete bipartite subgraph. For $s\in\mathbb{N}$ and $\beta\in\mathbb{R}^+$, a graph $G$ is an \defn{$(s,\beta)$-expander} if $G$ contains no $K_{s,s}$ subgraph, and for every $S \subseteq V(G)$ with $|S| \leq \frac{|V(G)|}{2}$ we have $\abs{N_G(S)} \ge \beta \abs{S}$. 
Expanders can be constructed probabilistically or constructively; see the survey \citep{HLW06} for example. In the following result it is necessary to exclude some fixed complete bipartite subgraph, as otherwise complete bipartite graphs would be counterexamples. 

\begin{thm}
\label{ExpanderClass}
Let $\GG$ be an infinite class of $(s,\beta)$-expanders for some $s\in\mathbb{N}$ and $\beta\in\mathbb{R}^+$. 
Then $\GG$ has unbounded $\reduced{\bw}$.
\end{thm}

\begin{proof}
Assume for contradiction that there exists $k\in\mathbb{N}$ such that $\reduced{\bw}(G) \le k$ for every $G \in \GG$. Consider an $n$-vertex graph $G\in\GG$ where $n\gg k,s,\beta^{-1}$ (as detailed below). 

Let $G=G_n, G_{n-1}, \dots, G_1$ be a reduction sequence for $G$ such that every red graph has bandwidth at most $k$. It is convenient to consider the corresponding partition sequence 
  \[\{\{v\} : v \in V(G)\}=\cP_n, \cP_{n-1}, \dots, \cP_1=\{V(G)\},\]
  and the associated trigraphs $G_{\cP_n}, G_{\cP_{n-1}}, \dots, G_{\cP_1}$ isomorphic to $G_n, G_{n-1}, \dots, G_1$ respectively. If $G_{i-1}=G_i/u,v$ and $u\in X$ and $v\in Y$, then $X$ and $Y$ are distinct parts of $\cP_i$, and $\cP_{i-1}$ is obtained from $\cP_i$ by replacing $X$ and $Y$ by $X\cup Y$. We consider the parts of $\cP_i$ to be the vertices of $G_{\cP_i}$. Say a part $X\in \cP_i$ is \defn{big} if $|X|\geq s$ and \defn{small} otherwise. Observe that for each black edge $XY$ in $G_{\cP_i}$, there is a complete bipartite subgraph in $G$ between $X$ and $Y$, which implies that $X$ or $Y$ is small. Moreover, if $X$ is big and $Y_1,\dots,Y_r$ are the neighbours of $X$ for which $XY_i$ is black, then there is a complete bipartite subgraph in $G$ between $X$ and $Y_1\cup\dots\cup Y_r$, implying $|Y_1\cup\dots\cup Y_r|\leq s-1$. 

Consider the first trigraph $G_{\cP_m}$ of the sequence (that is, the largest $m$) such that $\cP_m$ contains a part $X_0$ of size at least $\sqrt n$. Thus $X_0$ is big (for large $n$). Let $x:=\abs{X_0}$ and $t := \lceil \sqrt x \rceil$. By the choice of $m$, we have $\sqrt n \le x \le 2 \sqrt n$, and every part in $\cP_m$, except $X_0$, has size less than~$\sqrt{n}$. 
  
  We now show there are distinct parts $X_0,X_1,\ldots,X_t \in \cP_m$ such that $\abs{X_i} \ge \frac{\beta x}{4ki}$ for each $i\in\{1,\dots,t\}$, and $\{X_0,X_1,\dots,X_t\}$ induces a connected red subgraph in $G_{\cP_m}$. Assume that $X_0,X_1,\dots,X_i$ satisfy these properties for some $i\in\{0,\dots,t-1\}$. We now show how to find $X_{i+1}$. For $j\in\{1,\dots,i\}$,
\[|X_j|
\geq\frac{\beta x}{4kj}
\geq\frac{\beta x}{4k(t-1)} 
\geq\frac{\beta x}{4k\sqrt{x}} 
\geq\frac{\beta \sqrt{x}}{4k} 
\geq\frac{\beta n^{1/4}}{4k} 
\geq s \text{ (for large $n$)}.\] 
So $X_0,X_1,\dots,X_i$ are all big. 
Let $Y_1,\dots,Y_p$ be the parts of $\cP_m \setminus\{ X_0,\dots,X_i\}$ joined to $\{X_0, \ldots, X_i\}$ by black edges. 
As argued above, $|Y_1\cup\dots\cup Y_p| \leq (i+1)(s-1)$. 
Let $Z_1,\dots,Z_q$ be the parts of $\cP_m \setminus\{ X_0,\dots,X_i\}$ joined to $\{X_0, \ldots, X_i\}$ by red edges. 
Each of $X_0,\dots,X_i$ are incident to at most $2k$ red edges in $G_{\cP_m}$ since graphs of bandwidth $k$ have maximum degree at most~$2k$. So $q\leq 2k(i+1)$. 
On the other hand, as $G$ is an $(s, \beta)$-expander and
\[\abs{X_0\cup\dots\cup X_i} \leq (i+1)x \leq tx \le \frac{n}{2},\] we have
\[|Y_1\cup\dots\cup Y_p|+|Z_1\cup\dots\cup Z_q| \geq |N_G(X_0\cup\dots\cup X_i)|\geq \beta x\] and
\[|Z_1\cup\dots\cup Z_q|\geq \beta x-(i+1)(s-1) \geq\frac{\beta x}{2}\] (for large $n$).  
Let $X_{i+1}$ be the largest of $Z_1,\dots,Z_q$. 
So $|X_{i+1}|\geq \frac{\beta x}{2q} \geq \frac{\beta x}{4k(i+1)}$, 
and $\{X_0,X_1,\dots,X_{i+1}\}$ induces a connected red graph in $G_{\cP_m}$, as desired. Hence there are distinct parts $X_0,X_1,\ldots,X_t \in \cP_m$ such that $\abs{X_i} \ge \frac{\beta x}{4ki}$ for each $i\in\{1,\dots,t\}$, and $\{X_0,X_1,\dots,X_t\}$ induces a connected red subgraph in $G_{\cP_m}$. As shown above, $X_0,\dots,X_t$  are all big. 

Let $H$ be the connected component of $G_{\cP_m}$ containing $X_0,\dots,X_t$. Consider a vertex-ordering of~$H$ with bandwidth at most~$k$. Since $\{X_0, X_1, \ldots, X_t\}$ induces a connected red subgraph of~$H$, there is a set $I$ of $kt+1$ consecutive vertices in this ordering of $H$ containing $X_0,\dots,X_t$ and with at most $k-1$ parts strictly between `consecutive' $X_i$s. 

Let $Z$ be the union of all the big parts in~$I$. 
So $\abs{Z} \le (kt+1)x \le \frac{n}{2}$ (for large $n$), implying 
\[\abs{N_G(Z)}\geq \beta \abs{Z}
\ge \beta \abs{X_1\cup\dots\cup X_t} 
\ge \frac{\beta^2 x}{4k} \sum_{i=1}^t\frac{1}{i} 
\ge \frac{\beta^2 x \ln t}{4k} 
.\]
We now upper bound $|N_G(Z)|$. Each vertex in $N_G(Z)$ is either in: 
\begin{itemize}
    \item a small part of $\cP_m$ in $I$, 
    \item a part of $\cP_m$ adjacent to $I$ in $H$, or 
    \item a part of $\cP_m$ outside $H$.
\end{itemize}
There are at most $(kt+1)(s-1)$ vertices in small parts of $\cP_m$ in $I$. There are at most $2k$ parts of $\cP_m$ adjacent to $I$ in $H$ ($k$ to the left of $I$, and $k$ to the right), and each such part has at most $x$ vertices, thus contributing at most $2kx$ vertices to $N_G(Z)$. Each big part in~$I$ is incident to at most $s-1$ vertices in parts outside $H$ (since the corresponding edges are black), so there are at most  
$(kt+1)(s-1)$ vertices in $N_G(Z)$ of the third type. 
Hence
\[ \frac{\beta^2 x\ln t}{8k} 
\leq |N_G(Z)| 
\leq (kt+1)(s-1)  + 2kx + (kt+1)(s-1). \]
This is the desired contradiction, since  the left-hand side is $\Theta(\sqrt{n} \log n)$ and right-hand side is $\Theta(\sqrt{n})$ for fixed $k,s,\beta$.
\end{proof}

\section{Tied and Separated Parameters}
\label{Tied}

Let $\preceq$ be the preorder, where for graph parameters $f_1$ and $f_2$, define \defn{$f_1\preceq f_2$} if there exists a~function~$g$ such that $f_1(G)\leq g( f_2(G))$ for every graph $G$. Graph parameters $f_1$ and $f_2$ are \defn{tied} (also called \defn{functionally equivalent}) if $f_1\preceq f_2$ and $f_2\preceq f_1$. For example, \citet{TW-VI} showed that total twin-width is tied to linear rank-width, while component twin-width,  clique-width and boolean-width are tied. 

Define \defn{$f_1 \precneq f_2$} to mean $f_1\preceq f_2$ and  $f_2\not\preceq f_1$. For example, $\tww \precneq \reduced{\bw}$ since for every graph~$G$, we have $\Delta(G)\leq 2 \bw(G)$ implying $\tww(G)\leq 2 \,\reduced{\bw}(G)$, but there are bounded degree expanders with twinwidth 6 \citep{TW-II} and unbounded reduced bandwidth by \cref{ExpanderClass}. 
Graph parameters $f_1$ and $f_2$ are \defn{separated} if 
$f_1\precneq f_2$ or $f_2\precneq f_1$.

Recall that in the context of reduction sequences, it only makes sense to consider graph parameters that are unbounded on stars. This motivates the following definition. For a graph parameter $f$ let \defn{$f+\Delta$} be the graph parameter defined by $(f+\Delta)(G):=f(G)+\Delta(G)$. This section shows that $\reduced{\Delta}$, $\reduced{(\Delta+\tw)}$, $\reduced{(\Delta+\pw)}$ and $\reduced{\bw}$ are separated (even within graphs).

\begin{thm}\label{par-separation}
$\reduced{\Delta} \precneq \reduced{(\Delta+\tw)} \precneq \reduced{(\Delta+\pw)} \precneq \reduced{\bw}$.
\end{thm}

The proof of \cref{par-separation} uses a lemma of \citet{BBD21,BBD22}. 

For every graph $G$, let \defn{$\red(G)$} be the trigraph $G'=(V(G')=V(G),E(G')=\emptyset,R(G')=E(G))$ obtained by making all its edges red. An \defn{induced subtrigraph} of~$G$ is another trigraph~$H$ obtained by removing some vertices from $G$, and $G$ is then called a \defn{supertrigraph} of~$H$. Note that $\reduced{f}(G)$ is well-defined for any graph parameter $f$ and any trigraph $G$ (since reduction sequences are  defined for trigraphs $G$).

The \defn{2-blowup} of a graph $G$, denoted by \defn{$G \blowup K_2$}, is the graph obtained by replacing each vertex $u \in V(G)$ by two non-adjacent vertices $u_1, u_2$, and replacing each edge $uv \in E(G)$ by four edges $u_1v_1, u_1v_2, u_2v_1, u_2v_2$.

\begin{lem}[essentially Lemma~10 in \citep{BBD21}]
\label{lem:trigraph-to-graph}
For every connected graph~$H$ of maximum degree~$d$, there is a graph $G$ admitting a partial reduction sequence $G_n, \ldots, G_i$ such that:
\begin{itemize}
    \item  $\Delta(\widetilde{G_k}) \leq 2d$ for every $k \in \{i,i+1,\ldots,n\}$, 
    \item each connected component of $\widetilde{G_k}$ is a subgraph of $G \blowup K_2$,
    \item $G_i$ is isomorphic to $\red(H)$, and
    \item every (full) reduction sequence of $G$ goes through a supertrigraph of $\red(H)$ or a~trigraph with red maximum degree at least $|V(H)|$.
\end{itemize}
\end{lem}
\begin{proof}[Proof Sketch]
In Lemma~10 of~\cite{BBD21}, the second item does not appear, and the second option of the fourth item consists of going through a trigraph of red maximum degree at least $2d+1$.
The lemma is actually shown in greater generality (see Lemma 11 of~\cite{BBD21}) and accounts for trigraphs possibly having black edges and a disconnected red graph. Since we do not need this general form here, we can simplify the construction.

The graph $G$ is built as follows, where 
$t := \abs{V(H)}^{\abs{V(H)}}+\abs{V(H)}$. 
For every vertex $v \in V(H)$, add to $G$ a clique $v_1, \ldots, v_t$.
We informally refer to  $\{\{v_1,\ldots,v_t\} : v \in V(H)\}$ as the \defn{cliques of~$G$}.
For each edge $vw \in E(H)$, add to $E(G)$ the matching $v_1w_1, v_2w_2, \ldots, v_tw_t$. 

We now describe the partial reduction sequence of $G$ satisfying the first three items.
For every $v \in V(H)$, identify $v_1$ and $v_2$, and call the resulting vertex $v'_2$.
Then for every $v \in V(H)$, identify $v'_2$ and $v_3$, and call the resulting vertex $v'_3$; and so on, until every clique of every $v \in V(H)$ has been identified to a single vertex. It is straightforward to check that this partial reduction sequence satisfies the first three items.

The proof of the fourth item follows the proof of Statement~5 in Lemma~11 in \citep{BBD21}. We sketch the argument here, since the situation is simpler. If two parts of $G$ intersecting different cliques of $G$ are identified, the red degree of the resulting part is at least the number $s$ of other parts intersecting these two cliques, possibly minus 2.
Thus the second condition of the item is satisfied unless $s < \abs{V(H)}+2$.
We may now assume that immediately  before the first such identification, there is a part $P_v$ of size at least $\abs{V(H)}^{\abs{V(H)}-1}$ entirely contained within the clique of $v \in V(H)$.
Let $w$ be a neighbour of $v$ in $H$. 
For the red degree of $P_v$ to be less than $\abs{V(H)}$, the $\abs{P_v}$ matched vertices in the clique of $w$, have to be in less than $\abs{V(H)}$ parts.
This means that the largest of those parts, $P_w$, has size at least $\abs{V(H)}^{\abs{V(H)}-2}$.
Since $H$ is connected, we continue this reasoning along a spanning tree of $H$, and exhibit a collection $\{P_u : u \in V(H)\}$ of parts inducing the trigraph $\red(H)$.
\end{proof}

The next lemma says that, for graph parameters $f$ and $g$ satisfying certain properties, $f \precneq g$ implies $\reduced{f} \precneq \reduced{g}$. A~parameter $f$ is \defn{2-blowup-closed} if there is a function $g: \mathbb N_0 \to \mathbb N_0$ such that $f(G \blowup K_2) \leq g(f(G))$ for every graph~$G$.

\begin{lem}\label{lem:par-to-reduced-par}
Let $f$ and $g$ be monotone graphs parameters such that:
\begin{itemize}
\item $f \precneq g$,
\item $f$ is union-closed and \mbox{2-blowup-closed}, 
\item there is a non-decreasing function $q': \mathbb N_0 \to \mathbb N_0$ satisfying $\lim_{n\to\infty} q'(n) = \infty$ and $g(G) \geq q'(\Delta(G))$ for every graph $G$, 
and
\item 
there is a constant $c \in \mathbb N$ such that for every $n \in \mathbb N$ there is a connected graph $H_n$ on at least $n$ vertices with $\lim_{n\to\infty} g(H_n) =\infty$ and $\reduced{f}(\red(H_n)) \leq c$.
\end{itemize}    
Then $\reduced{f} \precneq \reduced{g}$.
\end{lem}
\begin{proof}
For a graph~$H$, let $t(H)$ be the graph $G$ obtained from \cref{lem:trigraph-to-graph}. Define $\mathcal F := \{t(H_n) : n \in \mathbb{N}\}$. 
Since $f \preceq g$, it is immediate that $\reduced{f} \preceq \reduced{g}$. To prove that~$\reduced{g} \not\preceq \reduced{f}$, we show that $\reduced{f}$ is bounded on $\mathcal F$, but  $\reduced{g}$ is unbounded. 

We first show that $\reduced{f}$ is bounded on $\mathcal F$.
Let $G =t(H) \in \mathcal F$. By~\cref{lem:trigraph-to-graph}, there is a partial reduction sequence $G_{\abs{V(G)}}, \ldots, G_i$ of $G$ such that $G_i$ is isomorphic to $\red(H)$, and every red graph $\widetilde{G_k}$ (for $k \in \{i, \ldots, \abs{V(G)}\}$) is a disjoint union of subgraphs of the 2-blowup of $H$.
Since $\reduced{f}(\red(H)) \leq c$, in particular, $f(H) \leq c$.
Since $f$ is monotone, union-closed and 2-blowup-closed, there is a function $q: \mathbb N_0 \to \mathbb N_0$ such that $f(\widetilde{G_k}) \leq f(H \blowup K_2) \leq q(f(H)) \leq q(c)$.
By assumption $\reduced{f}(G_i) \leq c$, thus $G$ has a (full) reduction sequence witnessing that $\reduced{f}(G) \leq \max(q(c),c)$.

We now show that $\reduced{g}$ is unbounded on $\mathcal F$.
For the sake of contradiction, suppose there exists $c' \in \mathbb N$ such that $\reduced{g}(G') \leq c'$ for every $G' \in \mathcal F$.
Let $n \in \mathbb N$ be such that $g(H_n)>c'$ and $q'(n) > c'$. 
By assumption such an integer always exists.
Consider $G :=t(H_n) \in \mathcal F$. 
By~\cref{lem:trigraph-to-graph}, every reduction sequence of $G$ either goes through a~trigraph $Y$ with red maximum degree \mbox{$\abs{V(H_n)} \geq n$} or through a~supertrigraph $Z$ of $\red(H_n)$.
In the former case, $g(\widetilde{Y}) \geq q'(\Delta(\widetilde{Y})) \geqslant q'(n) > c'$.
In the latter case, since $g$ is monotone, $\reduced{g}(Z) \geq \reduced{g}(\red(H_n)) \geq g(H_n) > c'$.
Both cases imply that $\reduced{g}(G) >  c'$, which contradicts the boundedness of $\reduced{g}$ on~$\mathcal F$.
\end{proof}

\begin{proof}[Proof of \cref{par-separation}]
First observe that $\Delta$, $\Delta+\tw$, $\Delta+\pw$ and $\bw$ are monotone, union-closed and \mbox{2-blowup-closed} with function $n \mapsto 2n+1$. If $g$ is any of these parameters, then 
$g(G) \geq \lfloor\frac{\Delta(G)}{2}\rfloor$, which provides the $q'$ function in \cref{lem:par-to-reduced-par}.

The first claim, $\reduced{\Delta} \precneq \reduced{(\Delta+\tw)}$, follows from  \cref{lem:par-to-reduced-par} where $H_n$ is the 
$n\times n$ planar grid graph, which has $(\Delta+\tw)(H_n) \geq n$ (see \citep{HW17}) and $\reduced{\Delta}(\red(H_n))=4$ (see \citep{TW-I}).

The second claim, $\reduced{(\Delta+\tw)} \precneq \reduced{(\Delta+\pw)}$, follows from  \cref{lem:par-to-reduced-par} where $H_n$ is the complete binary tree of height $n$, which has $(\Delta+\pw)(H_n) \geq \frac{n}{2} $ (see \citep{Scheffler89}) and $\reduced{(\Delta+\tw)}(\red(H_n))\leq 4$ (see~\citep{TW-I}).

Finally, $\reduced{(\Delta+\pw)} \precneq \reduced{\bw}$ follows from  \cref{lem:par-to-reduced-par} where $H_n$ is the tree defined in \cref{QnHere} on page~\pageref{QnHere}, which has $\bw(H_n)\geq \frac{n}{3}$ and $\reduced{(\Delta+\pw)}(\red(H_n))\leq 5$. The latter can be seen by iteratively identifying any leaf with its parent (in any order), which yields only subtrigraphs of $H_n$ and can be done until the trigraph has a single vertex. Throughout, the red graphs have maximum degree at most~3 and pathwidth at most~2.
\end{proof}

\section{Open Problems}
\label{OpenProblems}

Our results lead to several interesting open problems. 

\paragraph{Parameter tied to reduced bandwidth?} 

Recall that component twin-width and clique-width are tied~\citep{TW-VI}. Is there a `natural' graph parameter tied to reduced bandwidth? This is related to the question of which dense graph classes have bounded reduced bandwidth. Our results for powers give such examples. Complements provide other examples: if $\overline{G}$ is the complement of a graph $G$, then $\reduced{\bw}(\overline{G})=\reduced{\bw}(G)$, since any reduction sequence for $G$ defines a reduction sequence for $\overline{G}$ with equal red graphs. In general, $\reduced{f}(G)=\reduced{f}(\overline{G})$ for every $G$ and~$f$. 
\citet{TW-III} proved that unit interval graphs have  twin-width~2; in fact, they have a reduction sequence in which every red graph is a disjoint union of paths. Thus unit interval graphs have reduced bandwidth 1.

\paragraph{Best possible parameters:}
Is bandwidth the largest possible parameter $f$ such that planar graphs have bounded $\reduced{f}$? We formalise this question as follows. Say a graph parameter $f$ is \defn{continuous} if there exists a function $g$ such that
$f(G) \leq g( f(G-e) )$ for every graph $G$ and edge $e\in E(G)$, and
$f(G) \leq g( f(G-v) )$ for every graph $G$ and isolated vertex $v$ of $G$. 
All the graph parameters studied in this paper are continuous. 
Is there a continuous\footnote{We need to assume $f$ is continuous because of the following example. For a graph $G$, define $f(G):=\bw(G)$ if $\bw(G)\leq 466$ and $f(G):=|V(G)|$ otherwise. So $\bw\precneq f$ and planar graphs have bounded $\reduced{f}$, but $f$ is not continuous.} graph  parameter $f$ such that $\bw \precneq f$ and planar graphs have bounded $\reduced{f}$?  Maximum component size is not such a parameter, since planar grid graphs have unbounded clique-width~\citep{GR00} and therefore have unbounded $\reduced{\star}$. In general, for a graph class $\GG$, what is a continuous graph parameter $f$ such that $\reduced{f}$ is bounded on $\GG$, and $\reduced{g}$ is unbounded on $\GG$ for every parameter $g$ with $f \precneq g$. Does such a graph parameter always exist?

\paragraph{Lower Bounds:}

A non-trivial lower bound on reduced bandwidth is \cref{ExpanderClass} for expanders. Stronger lower bounds are plausible. For example, does every monotone graph class excluding a fixed complete bipartite subgraph and with bounded reduced bandwidth have polynomial expansion? It is even plausible that the degree of the polynomial is an absolute constant. Does every monotone graph class excluding a fixed complete bipartite subgraph and with bounded reduced bandwidth have linear expansion? Graphs with bounded row-treewidth have linear expansion. A good example to consider is the class of 3-dimensional grids, which we guess have unbounded reduced bandwidth.

\paragraph{Subdivisions:}

\citet{TW-II} showed that $t$-subdivisions of $K_n$ have bounded twin-width if and only if $t \in\Omega(\log n)$. No explicit bound on the twin-width was given. Recently, \citet{BBD21} proved that every $(\geq 2 \log n)$-subdivision of any $n$-vertex graph has twin-width at most 4. The proof  constructs a reduction sequence in which every red graph is a forest. So this class of graphs has exponential expansion, is  $K_{2,2}$-free,  and has  $\reduced{(\tw+\Delta)}\leq 5$. The proof in \citep{BBD21} can be adapted\footnote{In the proof of Theorem~7 in~\citep{BBD21}, replace the virtual binary red tree by an $n$-vertex red path, and observe that running the same sequence produces a red caterpillar with maximum degree 4 and at most one vertex of degree~4; refer to Figures~3 and 4 in~\cite{BBD21}. Such caterpillars have bandwidth at most~2.} to show that every $(\geq n)$-subdivision of any $n$-vertex graph has reduced bandwidth at most 2. Is this result best possible, in the sense that if $K_n^{(p)}$ is the $p$-subdivision of $K_n$ and $\reduced{\bw}(K_n^{(p)})\leq O(1)$, must $p\in\Omega(n)$? Since $\Omega(n)$-subdivisions of $n$-vertex graphs have linear expansion, a positive answer to this question would be good evidence for the suggestion above that graph classes with bounded reduced bandwidth and excluding a fixed complete bipartite subgraph have linear expansion.

\subsection*{Acknowledgements}
This research was initiated at the Dagstuhl workshop \emph{Sparsity in Algorithms, Combinatorics and Logic} (September 2021). Many thanks to the organisers and other participants. 

\subsection*{Postscript}

After the initial announcement of our results, \citet{JP22} proved that planar graphs have twin-width at most 183. Their proof is based on the proof of \cref{PlanarProductStructure}. These methods have been subsequently tailored to the twin-width setting, leading to better bounds \citep{HJ25b,Hlineny25}. The best known upper bound on the twin-width of planar graphs is 8 due to \citet{HJ25b}; the best known lower bound is 7 due to \citet{KL25}. \citet{KPS24} used a refined product structure theorem to prove a $O(\sqrt{g})$ bound on the twin-width of graphs with genus $g$.

{
\small
\let\oldthebibliography=\thebibliography
\let\endoldthebibliography=\endthebibliography
\renewenvironment{thebibliography}[1]{%
\begin{oldthebibliography}{#1}%
	\setlength{\parskip}{0ex}%
	\setlength{\itemsep}{0ex}%
}{\end{oldthebibliography}}
\def\soft#1{\leavevmode\setbox0=\hbox{h}\dimen7=\ht0\advance \dimen7
  by-1ex\relax\if t#1\relax\rlap{\raise.6\dimen7
  \hbox{\kern.3ex\char'47}}#1\relax\else\if T#1\relax
  \rlap{\raise.5\dimen7\hbox{\kern1.3ex\char'47}}#1\relax \else\if
  d#1\relax\rlap{\raise.5\dimen7\hbox{\kern.9ex \char'47}}#1\relax\else\if
  D#1\relax\rlap{\raise.5\dimen7 \hbox{\kern1.4ex\char'47}}#1\relax\else\if
  l#1\relax \rlap{\raise.5\dimen7\hbox{\kern.4ex\char'47}}#1\relax \else\if
  L#1\relax\rlap{\raise.5\dimen7\hbox{\kern.7ex
  \char'47}}#1\relax\else\message{accent \string\soft \space #1 not
  defined!}#1\relax\fi\fi\fi\fi\fi\fi}

}

\end{document}